\documentclass[11pt,a4paper]{article}

\usepackage{geometry}
\usepackage{amsmath,amssymb,amsthm,mathtools,bm}
\usepackage{enumitem}
\usepackage{microtype}
\usepackage[colorlinks=true,linkcolor=blue,citecolor=blue,urlcolor=blue]{hyperref} 

\usepackage{tikz,bm}

\usetikzlibrary{shapes.geometric, arrows.meta, positioning, calc} 

\newtheorem{theorem}{Theorem}[section]

\newtheorem{lemma}[theorem]{Lemma}
\newtheorem{conjecture}[theorem]{Conjecture}
\newtheorem{problem}[theorem]{Problem}
\newtheorem{claim}[theorem]{Claim}

\theoremstyle{definition}
\newtheorem{definition}[theorem]{Definition}

\newtheorem{remark}{Remark}
\theoremstyle{plain}

\newcommand{\Aut}{\operatorname{Aut}}

\newcommand{\eps}{\varepsilon}

\newcommand{\Ksharp}{\kappa_{F}}

\renewcommand{\leq}{\leqslant}
\renewcommand{\le}{\leqslant}
\renewcommand{\geq}{\geqslant}
\renewcommand{\ge}{\geqslant}

\title{An edge-spectral supersaturation of Mubayi's theorem\\ for color-critical graphs}

\author{Hongzhang Chen\thanks{School of Mathematics and Statistics, Gansu Center for Applied Mathematics, Lanzhou University, Lanzhou, Gansu, 730000, China. Email: \url{mnhzchern@gmail.com}.} 
\and Yongtao Li\thanks{Corresponding author. 
Yau Mathematical Sciences Center (YMSC), Tsinghua University, Beijing, 100084, China. Email: \url{ytli0921@hnu.edu.cn}.} 
}

\date{}

\begin{document}
\maketitle

\vspace{-7mm}
\begin{abstract}
The supersaturation problem asks how many copies of a fixed graph $F$ are forced to appear in a host graph once it passes the Tur\'an threshold.  
A celebrated theorem of Mubayi states that for any color-critical graph $F$ with chromatic number $\chi (F)=r+1\ge 3$, there exists a constant $\delta_F >0$ such that if $n$ is sufficiently large and $1\le q \le \delta_F n$, then every $n$-vertex graph $G$ with $e(G)\ge e(T_{n,r}) +q$ contains at least $q\, c(n,F)$ copies of $F$, 
where $c(n,F)$ is the minimum number of copies of $F$ 
created by adding one edge inside a part of the Tur\'an graph $T_{n,r}$. 
Writing $f=|F|$, Mubayi's estimate gives $c(n,F)=\alpha_F n^{f-2}+O_F(n^{f-3})$ for a
constant $\alpha_F>0$. 

We study the supersaturation problem in its edge-spectral form. Let $\lambda(G)$ be the adjacency spectral radius of $G$. Nikiforov proved that every $K_{r+1}$-free graph $G$ with $m$ edges satisfies $\lambda (G)\le \sqrt{(1\!-\!1/r )2m}$. Recently, 
 Li, Liu and Zhang proved the same bound for every $F$-free graph $G$, where $F$ is any color-critical graph with $\chi(F)=r+1\ge4$, with equality only for regular complete $r$-partite graphs. It is then natural to ask how many copies of
$F$ are forced once $\lambda (G)$ exceeds this threshold. Fang, Lin and Zhai answered
this at the threshold itself, and conjectured that for any fixed $C>0$, the condition $\lambda (G)\ge \sqrt{(1\!-\!{1}/{r})2m} +C$  forces $\Omega\!\left(m^{(f-1)/2}\right)$ copies. 

In this paper, we answer this question with the best possible constant. 
Building on the proof framework of Fang, Lin and Zhai, 
we prove that for every
color-critical graph $F$ with $\chi(F)=r+1\ge4$, there exists $\delta_F>0$ such that if $m$
is sufficiently large, $0<q\le\delta_F\sqrt m$, and $G$ is an $m$-edge graph with
$\lambda^2(G)\ge 2\left(1-\tfrac1r\right)m+q$, then
\vspace{-2mm}
\[
  N_F(G)\ge\bigl(B_F-o(1)\bigr)\,q\,
  m^{{(f-2)}/{2}},
  \quad \text{where}~~
  B_F:=\tfrac{\alpha_F}{4}
   (\tfrac{2r}{r-1} )^{{f}/{2}}, \vspace{-1mm} 
\] 
and the constant $B_F$ is best possible. 
Our result can be viewed as an edge-spectral counterpart of Mubayi's theorem, 
since it converts the spectral surplus $q$ into a linear number of copies of $F$, and it solves the conjecture of Fang, Lin and Zhai in a stronger form. 
\end{abstract}



\section{Introduction}

The well-known Mantel theorem says that every $n$-vertex graph with no triangle contains at most $\lfloor n^2/4\rfloor$ edges.  
The Tur\'an-type problem asks for the maximum number of edges in a graph with no copy of $F$;  the supersaturation problem asks how many copies of $F$ are forced once the number of edges exceeds the Tur\'{a}n number.  Erd\H{o}s and Rademacher (see \cite{Erdos1964}) extended Mantel's theorem by showing that  
if $e(G)>  \lfloor n^2/4\rfloor$, 
then $G$ contains at least $ \lfloor {n}/{2}\rfloor$ triangles. 
In general,  Erd\H{o}s \cite{Erd1962a}
 showed that there exists a constant $\delta>0$ such that if  $n$ is sufficiently large and $1\le q<\delta n$ is an integer, then
 $e(G)\ge \lfloor n^2/4\rfloor +q$
 forces at least $q\lfloor {n}/{2}\rfloor$ triangles in $G$. 
Furthermore,  
 Erd\H{o}s proposed a conjecture to determine the exact value of $\delta$. 
Finally, 
Lov\'{a}sz and Simonovits \cite{LS1975,LS1983} proved that 
for any positive integer $q < {n}/{2}$,  
if $G$ is an $n$-vertex graph with 
$e(G)\ge \lfloor {n^2}/{4} \rfloor + q$,   
then $G$ contains at least $q \lfloor {n}/{2}\rfloor$ 
triangles. 
The problem of counting triangles  is  referred to as the Erd\H{o}s--Rademacher problem, 
which is regarded as a starting point of supersaturation in graph theory; see \cite{Rei2016, XK2021,LM2022-Erd-Rad,BC2023,LPS2020}.

A graph is called {\it color-critical} if it contains an
edge whose deletion reduces its chromatic number.
 This family of graphs, which includes cliques and odd cycles, plays a central role in the
development of extremal graph theory. 
Let $T_{n,r}$ be the  $n$-vertex complete balanced $r$-partite graph. 
Simonovits \cite{Sim1966} showed that
for every  color-critical  graph $F$ with chromatic number $\chi (F)=r+1 \ge 3$,
if $n$ is sufficiently large and 
 $G$ is an $n$-vertex graph containing no copy of $F$, then
$ e(G) \le e(T_{n,r})$, with equality if and only if $G=T_{n,r}$.
A breakthrough of Mubayi \cite{Mub2010} established a sharp supersaturation result beyond the Simonovits theorem for all color-critical graphs.

\begin{theorem}[Mubayi \cite{Mub2010}] 
\label{thm-Mub2010}
Let  $F$ be a color-critical graph
with $\chi (F)=r+1 \ge 3$. There exists $\delta =\delta (F)>0$ such that
if $n$ is sufficiently large, $1\le q \le \delta n$, and
$G$ is an $n$-vertex graph with
$$ e(G)\ge e(T_{n,r}) +q, $$  
then $G$ contains at least $q\cdot c(n,F) $ copies of $F$, where 
$c(n,F)$ denotes the minimum number of copies of $F$  in a  graph obtained from the Tur\'{a}n graph $T_{n,r}$ by adding one edge. 
\end{theorem}

We refer the interested reader to \cite{PY2017, KMP2020, MY2025, 
Rei2016, MWZ2026} for related developments.

\subsection{Spectral supersaturation}

The \textit{spectral radius} $\lambda(G)$ of $G$ is the maximum modulus of the eigenvalues of its adjacency matrix $A(G)$.  
The spectral Tur\'an problem 
asks for the maximum spectral radius of a graph with no copy of a given forbidden subgraph. 
For instance, an old theorem of Wilf \cite{Wil1986} states that  every $K_{r+1}$-free graph $G$ on $n$ vertices satisfies $\lambda(G) \le \left(1-\frac{1}{r} \right)n$. 
Another well-known theorem of Nikiforov \cite{Nikiforov} shows that every $K_{r+1}$-free graph $G$ with $m$ edges satisfies $\lambda^2(G) \le \left(1-\frac{1}{r} \right)2m$, with equality (see  \cite{Nik2006, Nik2009}) if and only if $G$ is a complete bipartite graph when $r=2$, or a regular complete $r$-partite graph when $r\ge 3$; see \cite{KN2014,CY2026} for related extensions.  
Both of these spectral bounds imply the Tur\'{a}n bound by invoking $\lambda (G)\ge \frac{2m}{n}$.  
Recently, Li, Liu and Zhang~\cite{LLZ-edge-spectral}
extended the Erd\H{o}s--Stone--Simonovits  theorem to the edge-spectral setting:
if  $F$ is a fixed graph with $\chi(F)=r+1\ge 3$, then every $m$-edge
$F$-free graph $G$ satisfies $\lambda^2(G)\le \left(1-\frac1r+o(1)\right)2m$. 
Furthermore, Li, Liu and Zhang \cite{LLZ-edge-color-critical} extended Nikiforov's theorem to color-critical graphs, 
proving that if $F$ is a color-critical graph with $\chi (F)=r+1\ge 4$, then every $F$-free graph $G$ with sufficiently large size $m$ satisfies \begin{equation} \label{eq-LLZ}
\lambda^2(G)\le \Big( 1-\frac{1}{r} \Big)2m,
\end{equation}
with equality if and only if $G$ is a regular complete $r$-partite graph. In addition, 
Li, Liu and Zhang \cite{LLZ-edge-color-critical} established Tur\'{a}n-type results for color-critical graphs $F$ with $\chi (F)=3$ and some classical bipartite graphs, in which the spectral extremal graphs are often nearly split graphs. There is now a substantial literature on edge-spectral Tur\'{a}n-type results; we refer to \cite{ZLS2021,LZS2024,LLZ2025,LZZ2024,LLLY2026,ZLL2026}.

\smallskip 
Following the line of classical supersaturation, it is natural to investigate the spectral supersaturation, which asks how
many copies of $F$ must appear in a host graph $G$ when  $\lambda(G)$ exceeds the maximum possible spectral radius of an $F$-free graph.  
There are two ways to measure how far 
$\lambda (G)$ lies above the Tur\'an threshold. In the vertex-spectral version, 
one asks how many copies of $F$ are forced once
$\lambda(G)$ passes $(1-\frac1r )n$; 
see, e.g., \cite{LFP2024, FLLM2025, FLL2025}.  
 In the edge-spectral version, 
one asks the same question once $\lambda(G)$ passes $\sqrt{(1\!-\! 1/ r)2m}$.
The edge-spectral version is more general since  
the spectral scale $\sqrt m$ (in the case $r=2$) applies to sparse graphs of any edge density, and the assumption $ \lambda (G) >\sqrt{(1\!-\!1/r)2m}$ is weaker than $\lambda (G) > (1- \frac{1}{r})n$. In the edge-spectral version, 
the order $n$ is not given in advance, 
so the structure of $G$ must be recovered by edge-spectral arguments, via the edge-spectral forms of supersaturation and stability.

\begin{table}[htbp]
\centering
\renewcommand{\arraystretch}{1.25}
\begin{tabular}{@{}l p{0.22\textwidth} p{0.23\textwidth} p{0.2\textwidth}@{}}
\hline
 & {Classical version} &{Vertex-spectral} 
   &{Edge-spectral} \\
\hline
Fixed parameter
 & $n$ vertices
 & $n$ vertices
 &  $m$ edges \\[2pt]
Extremal bound
 & Tur\'an \cite{Turan1941} 
 \newline $e(G)\le (1-\frac{1}{r})\frac{n^2}{2}$
 & Wilf \cite{Wil1986} \newline 
 $\lambda(G)\le(1-\tfrac1r)n$
 &  Nikiforov~\cite{Nikiforov}  \newline $\lambda^2(G)\le (1-\tfrac1r)2m$ \\[2pt]
  Color-critical case & Simonovits~\cite{Sim1966}
 &  Nikiforov \cite{Nik2009EJC} 
 & Li--Liu--Zhang~\cite{LLZ-edge-color-critical} \\[2pt]
Supersaturation
 &   Mubayi~\cite{Mub2010} 
  & Fang--Li--Lin--Ma \cite{FLLM2025} 
 & Current paper  \\ 
\hline
\end{tabular}
\caption{Three different lines of supersaturation.}
\label{tab:three-lines}
\end{table}

\smallskip  
The goal of this paper is to investigate the edge-spectral supersaturation problem for color-critical graphs. 
 In 2007, Bollob\'{a}s and Nikiforov \cite{BN2007jctb} proved that
every graph $G$ with spectral radius $\lambda$ contains at least
        $\frac13\lambda \bigl(\lambda^2-m\bigr)$
triangles; see also \cite[Lemma 7]{CFTZ20}. In 2023, 
Ning and Zhai~\cite{NZ2021,NZ2021b} proved that if 
$\lambda(G)>\sqrt m$, then $G$ contains at least $\lfloor \frac{1}{2} (\sqrt m-1)\rfloor$ triangles 
and at least
$\frac1{2000}m^2$ copies of $C_4$.  
Li, Liu and Zhang~\cite{LiLiuZhang26} later sharpened the $C_4$ count, showing that $G$ 
contains at least $\left(\frac{1}{8} - o(1) \right)m^2$ copies of $C_4$, and the constant $\frac{1}{8}$ is best possible. 
Chen, Li and Tang \cite{CLT2026} showed that every graph $G$ contains at least $m(\lambda -\sqrt{m}\,)$ triangles. In addition, 
they \cite{CLT2026} proved that $\lambda (G)> \sqrt{m}$ also forces at least $\left(\frac{1}{8} - o(1) \right)m$ copies of the kite $C_4^+$ (the $4$-cycle with a chord), and the constant $\frac{1}{8}$ is best possible. 
Recently, Li, Lin, Liu and Zhang \cite{LLLZ2026} established edge-spectral supersaturation results for classical bipartite graphs, including $K_{t,t}$ and $C_{2t}$, 
by developing spectral Sidorenko inequalities.

\smallskip 
Using the probabilistic method, Li, Liu and Zhang \cite{LiLiuZhang26} proved that $\lambda^2(G) > (1- \frac{1}{r})2m$ forces $\Omega_r(m^{(r-1)/2})$ copies of the clique $K_{r+1}$ for all $r\ge 2$. For a general color-critical graph $F$, 
by Mubayi's estimate  (Lemma~\ref{lem:Mub2010-c}), there exists a constant $\alpha_F>0$ depending on $F$ such that 
\begin{equation*} \label{eq-alpha}
  c(n,F)=\alpha_F \, n^{f-2}+O_F(n^{f-3}).
\end{equation*}  
A recent breakthrough due to   
Fang, Lin and Zhai \cite{FLZ} establishes the supersaturation result  beyond Li--Liu--Zhang's bound (\ref{eq-LLZ}), proving  that for any color-critical graph $F$ with order $f$ and $\chi (F)=r+1\ge 4$, if $m$ is sufficiently large and $G$ is an $m$-edge graph satisfying $\lambda^2(G) \ge (1- \frac{1}{r})2m$, then
 \begin{equation*} \label{eq-FLZ}
  N_F(G) \ge \left( \alpha_F \Big(\frac{2r}{r-1} \Big)^{\!\! \frac{f-2}{2}} - o(1) \right)m^{\frac{f-2}{2}}, 
  \end{equation*} 
unless $G$ is a regular complete $r$-partite graph. 
This bound on $N_F(G)$ is asymptotically tight. 
Furthermore, 
Fang, Lin and Zhai \cite{FLZ} proposed the following conjecture for the regime where the spectral radius exceeds the threshold by a constant additive gap. 

\begin{conjecture}[Fang--Lin--Zhai \cite{FLZ}] 
\label{conj-FLZ}
Let $ F $ be a color-critical graph with $ |F| = f $
 and $ \chi(F) = r + 1 \geq 4 $. For any fixed positive constant $ C $ and sufficiently large $ m $, 
$$ \lambda(G) \geq \sqrt{\Big(1 - \frac{1}{r}\Big)2m} +C \quad \Rightarrow \quad 
N_F(G) = \Omega\! \left(m^{\!\frac{f-1}{2}}\right). $$   
\end{conjecture}

A recent result of Chen, Li and Tang \cite{CLT2026} shows that for any real $C>0$, if $G$ is an $m$-edge graph with $\lambda (G)\ge \sqrt{m} +C$, then $G$ contains more than $C\, m$ triangles, and this bound is asymptotically tight 
as witnessed by the split graphs. 
This is an edge-spectral version of the Lov\'{a}sz--Simonovits theorem, and it establishes the analogous statement of Conjecture \ref{conj-FLZ} in the case of triangles. 

\subsection{Main results}

In this paper, we establish an 
edge-spectral counterpart of Mubayi's result in Theorem \ref{thm-Mub2010}. 

\begin{theorem} \label{thm:squared-gap}
Let $F$ be a color-critical graph of order $f$ with $\chi (F)=r+1\ge 4$. There exists $\delta_F>0$ such that if $m$ is sufficiently large, 
$0<q\le \delta_F \sqrt m$, and $G$ is an 
$m$-edge graph with 
\[
  \lambda^2(G) \ge \Bigl(1-\frac1r\Bigr) 2m+q, 
\]
then  $  N_F (G)\ge
  \big(B_{F}- o(1) \big)\,q\,m^{(f-2)/2}$, 
where $B_{F}:=\frac{\alpha_F}{4}\bigl(\frac{2r}{r-1}\bigr)^{f/2}$
is best possible. 
\end{theorem}

For each $r\ge 3$, 
the clique $K_{r+1}$ is color-critical with $\chi(K_{r+1})=r+1\ge4$,
so Theorem~\ref{thm:squared-gap} applies.  
After adding one edge $uv$ inside a part of the Tur\'{a}n graph $T_{n,r}$, every copy of $K_{r+1}$ through $uv$ is obtained by choosing one
vertex from each of the other $r-1$ parts, so 
$c(n,K_{r+1}) =(1+o(1))(n/r)^{r-1}$. Hence $\alpha_{K_{r+1}}=(1/r)^{r-1}$. 
Thus, the clique case of Theorem \ref{thm:squared-gap} sharpens the result of Li, Liu and Zhang~\cite{LiLiuZhang26},
who showed that $\lambda^2(G)>2(1-\tfrac1r)m$ forces $\Omega_r(m^{(r-1)/2})$ copies of
$K_{r+1}$.

Theorem~\ref{thm:squared-gap} 
shows that the count is linear in the spectral
surplus $q$. On the one hand, the constant
$B_F$ is the rate at which the spectral surplus
turns into copies of $F$: each unit of $q$
above the threshold forces
$(B_F-o(1))\,m^{(f-2)/2}$ copies, just as each
extra edge forces $c(n,F)$ extra copies in
Mubayi's theorem. 
On the other hand, the
theorem identifies the asymptotically extremal construction (Section~\ref{sec:sharpness}): the Tur\'an graph with a matching added
inside one part  achieves this bound, because adding a matching is the least efficient way to push the spectral radius above the Tur\'{a}n threshold; see, e.g.,  \cite{FLLM2025}. 
When we place Theorem~\ref{thm:squared-gap} inside the whole range
of the spectral surplus $q=\lambda^2(G)-(1-\frac1r)2m$, our theorem
covers the range $0<q\le\delta_F\sqrt m$, where the minimum number of copies of
$F$ grows linearly  with a sharp constant. At
the other end, once $q\ge\varepsilon m$ for a fixed $\varepsilon>0$, the edge-spectral supersaturation forces
$\Theta(m^{f/2})$ copies of $F$. The intermediate range
$\delta_F\sqrt m\le q\le\varepsilon m$ is not understood. 
We point out that the cut-off at the order $\sqrt m$ is not an artifact of our method: Remark~\ref{rem:fixed-gap-caveat} shows that the count is no longer linear once the gap $C$ is bounded away from zero. A similar phenomenon appears in the vertex-spectral setting, where Fang, Li, Lin and Ma \cite{FLLM2025} showed that the surplus must stay below order $\sqrt n$ for the count to remain linear. 

\smallskip 
Our starting point is the work of Fang,
Lin and Zhai~\cite{FLZ}, who settled the
threshold case $\lambda(G) \!>\! \sqrt{(1\!-\!1/r)2m}$ 
and proposed 
Conjecture~\ref{conj-FLZ}. 
We build on their framework and resolve their conjecture
with the sharp constant by showing the following stronger variant of Theorem \ref{thm:squared-gap}.

\begin{theorem}\label{thm:linear-small-gap}
Let $F$ be a color-critical graph with  order $f$ and $\chi(F)=r+1\ge4$, and let 
$ \Ksharp:=\alpha_F (\frac{2r}{r-1} )^{(f-1)/2}$.
For every $\eta>0$, there exists $C_0=C_0(F,\eta)>0$ such that for every real number $C$ with $0<C\le C_0$ and sufficiently large $m$, if $G$ is an $m$-edge graph with  
$$  \lambda(G)\ge \sqrt{\Big(1- \frac{1}{r}\Big)2m}+C, $$
then
\[
   N_F (G)\ge
  \big(1 - \eta \big) \Ksharp \,C m^{\frac{f-1}{2}}.
\]
Moreover, the coefficient $\Ksharp$ is best possible in the following sense: for every $\eta>0$, there is $C_1=C_1(F,\eta)>0$ such that  for every $0<C\le C_1$, there are arbitrarily large $m$ and $m$-edge graphs $G$ satisfying $\lambda(G)\ge\sqrt{\left( 1 \!-\! {1}/{r}\right)2m} +C$ and
$  N_F (G)\le (1+\eta)\Ksharp\, C m^{(f-1)/2}$.  
\end{theorem}

Theorem~\ref{thm:linear-small-gap} implies the following exact limit
\begin{equation}\label{eq:exact-limit}
  \lim_{C\to 0^+}\ \lim_{m\to\infty} 
  \frac{1}{C m^{(f-1)/2}} \min_{G} N_F (G) =\Ksharp,
\end{equation}
where the minimum is taken over all $m$-edge graphs $G$ with $\lambda(G)\ge\sqrt{(1\!-\!1/r)2m}+C$.

\begin{remark}\label{rem:large-gap}
Although Theorem~\ref{thm:linear-small-gap} is stated only for small gaps $C\le C_0$, it already resolves Conjecture~\ref{conj-FLZ} for every fixed $C>0$. Indeed, if $\lambda(G)\ge \sqrt{(1\!-\!1/r)2m}+C$ with $C>C_0$, then $\lambda(G)\ge \sqrt{(1\!-\!1/r)2m}+C_0$, so $ N_F (G)\ge(1-\eta)\Ksharp\,C_0\,m^{(f-1)/2}=\Omega(m^{(f-1)/2})$, as needed. Thus, 
Theorem \ref{thm:linear-small-gap} 
refines the conjectured bound $\Omega(m^{(f-1)/2})$ in both respects: it makes the dependence on $C$ linear, and it gives the best possible constant $\kappa_F$. 
\end{remark}

\begin{remark}
When $r=2$ and $F=K_3$, it was shown in \cite{CLT2026} that the
number of triangles forced is exactly linear for all $C>0$, with split
graphs extremal. For $r\ge3$, this is no longer so. The matching-added construction 
is optimal only to first order in the gap: adding a star inside a part reaches the same gap with fewer internal edges, hence fewer
copies, so the count is a nonlinear function of the
gap (see Remark~\ref{rem:fixed-gap-caveat}). Thus, the per-edge spectral cost depends
on the shape of the graph added inside a part, and the color-critical
case with $r\ge3$ is richer than the triangle case. 
\end{remark}

Next, we show that
Theorem~\ref{thm:linear-small-gap}
immediately implies
Theorem~\ref{thm:squared-gap}. Set 
$\tau_r(m):=\sqrt{2(1 \!-\!1/r)m}$. 
Let $\eta>0$, let $C_0=C_0(F,\eta/2)$ be the
constant of
Theorem~\ref{thm:linear-small-gap}, and set
$\delta_F=\min\{1,C_0\}$. Suppose that
$\lambda^2(G)\ge\tau_r(m)^2+q$ with
$0<q\le\delta_F\sqrt m$. Then
$\lambda(G)\ge\tau_r(m)+C$,  where 
\begin{equation*} 
  C:=\sqrt{\tau_r(m)^2+q}-\tau_r(m)
       =\frac{q}{\sqrt{\tau_r(m)^2+q}+\tau_r(m)}.
\end{equation*}
 Since
$C\le q/\tau_r(m)
   \le\delta_F/\sqrt{2(1-1/r)}\le C_0$,
Theorem~\ref{thm:linear-small-gap} applies
with parameter $\eta/2$:
\[
  N_F(G)\ge\left(1-\frac{\eta}{2}\right)
           \Ksharp\,C\,m^{(f-1)/2}.
\]
Since $q=o(m)$, we have
$C=(1+o(1))\tfrac{q}{2\tau_r(m)}$, so
\[
  N_F(G)\ge(1+o(1))\left(1-\frac{\eta}{2}\right)
   \frac{\Ksharp}{2\sqrt{2(1-1/r)}}
   \,q\,m^{(f-2)/2}.
\]
For sufficiently large $m$, we have 
$(1+o(1))(1-\tfrac{\eta}{2})\ge1-\eta$ and
$\tfrac{\Ksharp}{2\sqrt{2(1-1/r)}}
  =\tfrac{\alpha_F}{4}
   \bigl(\tfrac{2r}{r-1}\bigr)^{f/2}=B_F$.
Hence, we get $N_F(G)\ge(1-\eta)B_F\,q\,m^{(f-2)/2}$,
completing the proof of
Theorem~\ref{thm:squared-gap}.

\paragraph{Proof overview of Theorem~\ref{thm:linear-small-gap}.}
We prove the lower bound by contradiction. 
Suppose that $G$ is an $m$-edge graph with
$\lambda(G)\ge\sqrt{(1-1/r)2m}+C$ but with fewer than $(1-\eta)\Ksharp\,Cm^{(f-1)/2}$ copies
of $F$; in particular $ N_F (G)=o(m^{f/2})$. The argument 
contains four steps.

First, we regularize $G$ while keeping almost all of the spectral gap. Deleting
$\alpha$-light edges and $\beta$-deficient vertices one at a time does not decrease the
edge-spectral density $\Phi=\lambda/\sqrt{e}$, and the supersaturation theorem forces the
process to stop after fewer than $\delta m$ edges are removed. The surviving graph $H$
satisfies $e(H)>(1-\delta)m$ and 
$\lambda(H)\ge\tau_r(e(H))+(1-\theta)C$ (Claim~\ref{clm:regularization}). 

Second, since $ N_F (H)=o(h^{f/2})$, the stability theorem gives a partition
$V(H)=V_1\cup\cdots\cup V_r$ that is close to a balanced Tur\'an graph. We then refine the 
structure and show that 
the exceptional sets of low-degree and
high-internal-degree vertices are empty, so that  every vertex has small
internal degree and few missing cross-edges, and the Perron vector is almost uniform,
$x_v^2\le(1+o(1))/n$ (Claim~\ref{clm:clean-case}). 
The second part is standard and adapted from Fang--Lin--Zhai~\cite{FLZ}. 

Third, we count the copies created by the edges inside the parts. Each such class-edge
$e$ creates at least $(1-o(1))c(n,F)$ copies of $F$ whose only within-part edge is exactly $e$
(Claim~\ref{clm:sharp-local-count}), a first-order sharpening of the local per-edge count used at
the threshold. The families counted for distinct class-edges are disjoint, so it remains
to bound from below the number $p$ of class-edges.

The last step is the main new one: 
it converts the spectral gap $C$ into a sharp bound on the  number of class-edges (Claim~\ref{clm:sharp-class-edges}). Let $p$ be the number of class-edges. 
Deleting these $p$ edges leaves an
$r$-partite graph, whose spectral radius is at most $\tau_r(h-p)$ by Nikiforov's theorem. 
Combining the almost-uniformity of the 
Perron vector with   
 the bound on $\lambda(H)$ from  Claim~\ref{clm:regularization}, 
 we obtain $p\ge(1-o(1))Cn$. Summing the
per-edge count over the $p$ class-edges yields
$N_F(G)\ge(1-o(1))\Ksharp\,Cm^{(f-1)/2}$, against the assumption. The construction in
Section~\ref{sec:sharpness} shows that the constant $\Ksharp$ cannot be improved, 
and Remark~\ref{rem:fixed-gap-caveat} shows that
the count is different once the gap $C$ is bounded away from zero.

\paragraph{New ingredients.} 
 To prove Theorem~\ref{thm:linear-small-gap}, we must 
carry an additive gap $C$ through the whole argument and convert it into a lower bound on the number $p$ of class-edges. This needs three different arguments
that do not appear in Fang--Lin--Zhai~\cite{FLZ}: a regularization that \emph{preserves} the spectral gap
(Claim~\ref{clm:regularization}); a first-order-accurate version of the per-edge count (Claim~\ref{clm:sharp-local-count}); 
 the sharp conversion $p\ge(1-o(1))Cn$ from the spectral gap to the
number of class-edges (Claim~\ref{clm:sharp-class-edges}). 

 The pruning of
light edges and deficient vertices that keeps the edge-spectral
density non-decreasing was applied in  \cite{Niki2021,NZ2021b,FLZ}. 
 The new ingredient here is a
regularization that keeps almost all of the spectral gap, together with the exact conversion
of that gap into the number of edges inside the parts. This pair of steps turns an
additive spectral surplus into a sharp edge count, 
and we expect this approach to be useful 
for other spectral supersaturation problems.

\paragraph{Organization.}
  Section~\ref{sec:prelim} collects 
some tools: Mubayi's estimate for the per-edge count, the Erd\H{o}s--Simonovits supersaturation, 
the edge-spectral Tur\'an
theorem, the edge-spectral supersaturation and stability, and a regularization lemma. In Section~\ref{sec:main}, we prove Theorem~\ref{thm:linear-small-gap} by regularizing the graph with almost the same spectral
gap, passing to a clean stability partition, counting the copies of $F$ created by each
class-edge, and converting the spectral gap into a sharp lower bound on the number of class-edges.
In Section~\ref{sec:sharpness}, we show that the constant in Theorem~\ref{thm:linear-small-gap} is best possible. 
In Section~\ref{sec:concluding}, we conclude with several open
problems and directions for further work.

\section{Preliminaries}

\label{sec:prelim}

Recall that $c(n,F)$ denotes the minimum number of copies of $F$ in a graph obtained from $T_{n,r}$ by adding one edge inside one part. 
The following estimate is due to Mubayi~\cite{Mub2010}.

\begin{lemma}[Mubayi's estimate \cite{Mub2010}]\label{lem:Mub2010-c}
Let $F$ be a color-critical graph with $|F|=f$ and $\chi(F)=r+1\ge 3$.  Then there are constants $\alpha_F>0$ and $\beta_F>0$ such that, for all sufficiently large $n$,
\[
  \left|c(n,F)-\alpha_F n^{f-2}\right|<\beta_F n^{f-3}.
\]
In particular, for all sufficiently large $n$, $  \frac12\alpha_F n^{f-2}<c(n,F)<2\alpha_F n^{f-2}$. 
\end{lemma}

The following lemma counts the copies of $F$ in a 
nearly Tur\'{a}n graph with one edge added. 

\begin{lemma}[Mubayi~\cite{Mub2010}]\label{lem:near-balanced-one-edge}
Let $F$ be color-critical with $|F|=f$ and $\chi(F)=r+1$.  For each vector $ \mathbf n=(n_1,\ldots,n_r)$ with  $n_1+\cdots+n_r=n$, 
let $c (\mathbf n,F)$ be the minimum number of copies of $F$ in the complete $r$-partite graph with parts of sizes $n_1,\ldots,n_r$ after adding one edge inside a partite set.  If $0\le \xi < \frac{1}{3r}$ and 
$ \left|n_j-\frac nr\right|\le \xi n$ for all $j\in [r]$, 
 then  $  c (\mathbf n,F)\ge \big( 1-O_F(\xi) \big)\, c(n,F).$ 
\end{lemma}

\subsection{Edge-spectral supersaturation and stability}

The following classical result of Erd\H{o}s and Simonovits says
that once the edge count exceeds the
Tur\'an number by a positive fraction of
$n^2$, the graph already contains
$\Omega(n^f)$ copies of $F$. 

\begin{lemma}[Erd\H{o}s--Simonovits \cite{ES1983}] 
\label{lem:ES-supersaturation}
Let $F$ be a fixed graph with $|F|=f$ and $\chi(F)=r+1$.  For every $\eta>0$, there are constants $\delta=\delta(F,\eta)>0$ and $n_0=n_0(F,\eta)$ such that every graph $G$ on $n\ge n_0$ vertices with
$  e(G)\ge e(T_{n,r})+\eta n^2$ 
contains at least $  \delta n^f$ copies of $F$.
\end{lemma}

We need to use the following edge-spectral Tur\'an theorem~\cite{Nikiforov}.

\begin{theorem}[Nikiforov \cite{Nikiforov}]\label{thm:nikiforov}
If $G$ is a $K_{r+1}$-free graph with $m$ edges, then
\[
  \lambda^2(G)\le {\Big(1-\frac1r\Big)2m}.
\]
For $r\ge 3$, the equality holds only for regular complete $r$-partite graphs.
\end{theorem}

Li, Liu and Zhang~\cite{LLZ-edge-spectral} 
 proved the following two theorems 
 for $F$-free graphs $G$.  
 Here we state them
in a slightly more general form: $G$ need not be $F$-free but 
only assumed to satisfy $N_F(G)=o(m^{f/2})$, which is the form we need, 
since in our argument the host 
graph $G$ is not $F$-free but contains few copies of $F$.  
This general form was recently established by Fang, Lin and Zhai~\cite{FLZ} using Simonovits' progressive induction,  
and it also follows from the $F$-free case  by using the graph removal lemma.

\begin{theorem}[Edge-spectral supersaturation \cite{FLZ}] 
\label{thm:edge-spec-ESS}
Let $F$ be a graph of order $f$ with $\chi(F)=r+1\ge 3$, 
and let $G$ be a graph of sufficiently large size $m$ such that 
$N_F(G)=o(m^{f/2})$. Then 
$$ \lambda^2 (G)\le \Big(1-\frac{1}{r}+o(1) \Big)2m. $$ 
\end{theorem}

For two disjoint vertex sets $A,B$, 
we write $K_{A,B}$ for the complete bipartite graph on the parts $A$ and $B$. 
For a vertex set $C$, we write $T_{C,r}$ 
for an $r$-partite Tur\'{a}n graph on the vertex set $C$.  

\begin{theorem}[Edge-spectral supersaturation-stability \cite{FLZ}] 
\label{thm:edge-spectral-stability}
Let $F$ be a graph of order $f$ and $\chi (F)=r+1\ge 3$.
For every $\varepsilon >0$, there exist $\delta >0$ and $m_0$ such that if $G$ is a graph of size $m\ge m_0$ with $N_F(G)=o(m^{f/2})$ and $\lambda^2(G)\ge (1- \frac{1}{r} - \delta )2m$, then 
\begin{itemize}
\item[\rm (a)] When $r = 2$, there exist disjoint vertex sets $A, B \subseteq V(G)$ such that $d(G, K_{A, B}) \leq \varepsilon m$.

\item[\rm (b)]
 When $r \geq 3$, there exists a vertex set $C \subseteq V(G)$ such that $d(G, T_{C, r}) \leq \varepsilon m$. 
 \end{itemize} 
\end{theorem}

The following lemma says that if  $G$ is an $m$-edge graph with spectral radius bounded away from $\sqrt{m}$, then every coordinate of the Perron--Frobenius eigenvector of $G$ is $O(m^{-1/4})$. 

\begin{lemma}[Li--Liu--Zhang \cite{LLZ-edge-spectral}] \label{lem:LLZ} 
Let $G$ be a graph with $m$ edges and $\bm{x}=(x_v)_{v\in V(G)}$ be the unit Perron--Frobenius eigenvector of $G$.  
 If $\lambda^2 (G) \ge (1+\delta)m$ where $0<\delta \le 0.79$, then 
\[ \max\{x_v : v\in V(G)\} < {\delta^{-4}m^{-1/4} }. \] 
\end{lemma}

\subsection{Regularization by deleting edges and vertices}

We will employ an edge-deletion technique that removes edges whose endpoints have small product of Perron weights. 
This technique was developed by Nikiforov \cite{Niki2021} for finding books; it was subsequently applied by Ning and Zhai \cite{NZ2021b} for counting $4$-cycles.  
We start with the following definition. 

\begin{definition}[$\alpha$-light edge]
\label{def:alpha-light}
Let $m=e(G)\ge2$ and let $\alpha>0$. Fix a
non-negative unit Perron vector $\bm{x}$ of $G$. 
An edge
$uv\in E(G)$ is $\alpha$-light with respect to $\bm{x}$ 
if $  x_ux_v\le {\alpha}/{\sqrt m}$. 
\end{definition}

In addition, we need to use a vertex-deletion argument, which removes a deficient vertex with its incident edges. 
The following concept is inspired by the work of Fang, Lin and Zhai \cite{FLZ}.

\begin{definition}[$\beta$-deficient vertex]
\label{def:beta-deficient}
Let $0<\beta<1$, let $m=e(G)$, and fix a
non-negative unit Perron vector $\bm{x}$ of $G$. 
A vertex $u\in V(G)$ is $\beta$-deficient with
respect to $\bm{x}$ if $x_{u}^2\le\beta/2$, 
\[
  2m\,x_{u}^2\le(1-\beta)\,d(u)
  \qquad\text{and}\qquad
  d(u)\le\sqrt{\tfrac{2m}{1-1/r}}.
\]
\end{definition}

The {\it edge-spectral density} of a graph $G$ is defined as 
 $ \Phi(G) := {\lambda(G)}/ {\sqrt{e(G)}}. $  
The following regularization is essentially the same as the reductions used by Nikiforov \cite{Niki2021} by deleting light edges, and by Fang, Lin and Zhai \cite{FLZ} by deleting deficient vertices. This lemma shows that removing a light edge or a deficient vertex (with all its incident edges) increases the edge-spectral density. 

\begin{lemma}[Regularization lemma \cite{Niki2021, FLZ}] 
\label{lem:admissible}
Let  $G$ be an $m$-edge graph with no isolated vertices. Suppose $\lambda(G)^2\ge (1-\frac{1}{r}) 2m$, fix a
non-negative unit Perron vector $\bm{x}$, and let
$\eps >0$.
\begin{enumerate}[label=\textup{(\roman*)}]
\item If $uv$ is an $\alpha$-light  edge with respect to $\bm{x}$, where $2\sqrt2\,\alpha+\eps\le 
{1}/{2} $, 
and $G'$ is obtained by deleting $uv$ and
discarding isolated vertices, then
$\Phi(G')-\Phi(G)\ge \eps/m$.

\item If $u$ is a $\beta$-deficient vertex with respect to $\bm{x}$, where $\eps\le\beta/4$, and $G'$ is obtained by
deleting all edges at $u$, then
$\Phi(G')-\Phi(G)\ge\eps\,d_G(u)/m$.
\end{enumerate}
\end{lemma}

\begin{proof}
(i) Here $e(G')=m-1$. Since $\bm{x}$ is a unit vector,
\[
  \lambda(G')\ge \bm{x}^TA(G')\bm{x} =\lambda(G)-2x_ux_v,
\]
and discarding isolated vertices changes neither
$\lambda$ nor this bound. Hence
\[
  \Phi(G')-\Phi(G)
  \ge\lambda(G)\Bigl(\frac1{\sqrt{m-1}}
       -\frac1{\sqrt m}\Bigr)
   -\frac{2x_ux_v}{\sqrt{m-1}}.
\]
For the first term, since $  
\frac1{\sqrt{m-1}}-\frac1{\sqrt m} > \frac1{2m^{3/2}}$ 
and $\lambda(G)\ge\sqrt{(1-1/r)2m}$, this term is at
least $\frac{1}{2m}\sqrt{2(1-1/r)}$. 
For the second term, we have  
$\sqrt{m-1}\ge\sqrt m/\sqrt2$, and the
$\alpha$-light bound gives $x_ux_v\le\alpha/\sqrt m$.
Therefore, we have $  \frac{2x_ux_v}{\sqrt{m-1}}
  \le\frac{2\sqrt2\,x_ux_v}{\sqrt m}
  \le\frac{2\sqrt2\,\alpha}{m}$. 
Combining the two estimates, we get 
\[
  \Phi(G')-\Phi(G)
  \ge\frac{\sqrt{2(1\!-\!1/r)}}{2m}
     -\frac{2\sqrt2\,\alpha}{m}
  =\frac1m\Bigl(\frac{\sqrt{2(1\!-\!1/r)}}{2}
                -2\sqrt2\,\alpha\Bigr)
  \ge\frac{\eps}{m}. 
\]
(ii) Write $d:=d_G(u)$, so
$e(G')=m -d$. Note that $\lambda(G)x_{u} = \sum_{v\in N(u)} x_v$ and 
$$Q:=\sum_{ij\in E(G')}2x_ix_j 
= \sum_{ij\in E(G)}2x_ix_j - 2x_u \sum_{v\in N(u)} x_v 
=(1-2x_u^2)\lambda(G).$$ By
the Rayleigh quotient, we have 
\[
  \lambda(G')\ge\frac{Q}{1-x_u^2}
  =\lambda(G)\,\frac{1-2x_u^2}{1-x_u^2}.
\]
Hence
\[
  \Phi(G')-\Phi(G)
  \ge\frac{\lambda(G)}{\sqrt m\,\sqrt{m-d}}
   \left(\Bigl(1-\frac{x_u^2}{1-x_u^2}\Bigr)\sqrt m
        -\sqrt{m-d}\right).
\]
Using $\sqrt m-\sqrt{m-d}\ge \frac{d}{2\sqrt m}$, the
$\beta$-deficiency bound $2m x_u^2\le(1-\beta)d$, and
$x_u^2\le\beta/2$,
\[
  \Bigl(1-\frac{x_u^2}{1-x_u^2}\Bigr)\sqrt m-\sqrt{m-d}
  \ge\frac{d}{2\sqrt m}\cdot\frac{\beta- x_u^2}{1-x_u^2}
  \ge\frac{\beta d}{4\sqrt m}.
\]
Since $\lambda(G)\ge\sqrt{(1\!-\!1/r)2m}\ge\sqrt m$ and
$\sqrt{m-d}\le\sqrt m$, we get 
$ \Phi(G')-\Phi(G)\ge\frac{\beta d}{4m}
  \ge\eps\,\frac dm$. 
\end{proof}

\section{Proof of the sharp linear bound}

\label{sec:main}

In this section, we prove Theorem \ref{thm:linear-small-gap}. To start with, we set  
the constants below so that
\[
  0<\eps_0\ll\eps_1^2\le\eps_1\ll\eps_2\ll\theta\ll\eta,
  \qquad
  0<\delta\ll\theta.
\]
Here $\eta$ is the error allowed in the final bound; $\theta$ is the part of the spectral
gap that the regularization may lose, and $\delta$ the fraction of edges it may delete, so
both are small compared with $\eta$. The three scales $\eps_2\gg\eps_1\gg\eps_0$ control the
structure of the surviving subgraph, from coarse to fine: $\eps_2$ (via $\beta=2\eps_2$) gives
the clean structure that makes the per-edge count sharp; $\eps_1$ controls the balance of the
parts and the uniformity of the Perron vector, and must beat the $\eps_2$-scale degree
deficit; and $\eps_0$ is the stability quality, which is
used in the edge-to-vertex counts and is smaller than $\eps_1^2$. 

Let $\gamma=\gamma(F,r)>0$ be the constant from the coarse local-count Claim~\ref{clm:local-count}.  This claim is proved later, but the constant depends only on $F$ and $r$.  Finally, we choose
\[
  0<C_0=C_0(F,r,\eta)\le \min\left\{1,\tfrac{1}{100}\eps_1\gamma\right\}.
\]
Let $0<C\le C_0$ be any fixed real number.  All later lower thresholds for the size $m$ may depend on $F,r,\eta,C$ and on the above hierarchy, but not on the host graph $G$.

Suppose for the contradiction that there are counterexamples for arbitrarily large $m$.  Thus $G$ is an $m$-edge graph satisfying
$ \lambda(G)\ge \sqrt{(1\!-\! {1}/{r})2m} +C$, 
but
\begin{equation}\label{eq:sharp-counter}
  N_F(G)<
  (1-\eta)\Ksharp\,C m^{(f-1)/2}.
\end{equation}
Since $C$ is fixed, we have $  N_F(G)=O (m^{(f-1)/2})=o(m^{f/2}).$ 

Next, we summarize the key steps of the proof of 
Theorem \ref{thm:linear-small-gap} in Figure \ref{fig:proof-flow-thm-1-4}. 

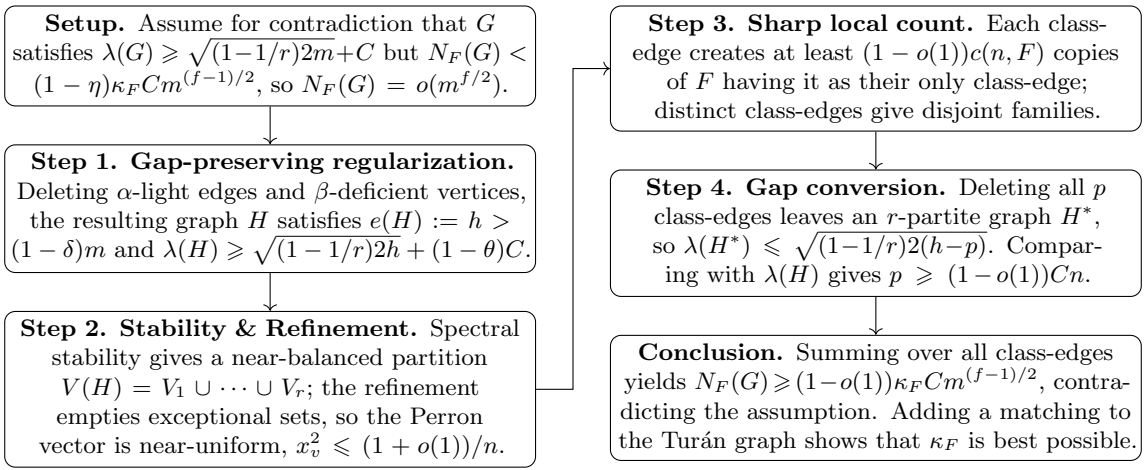
\begin{figure}[htbp]
\centering
\begin{tikzpicture}[
    node distance=5mm and 12mm,
    box/.style={rectangle, rounded corners, draw, align=center,
                text width=68mm, inner sep=3pt, font=\footnotesize},
    arr/.style={-{Stealth[length=1.8mm]}, semithick}
  ] 
\node[box] (A) {%
\textbf{Setup.} 
  Assume for contradiction that $G$ satisfies 
  $\lambda(G)\ge\sqrt{(1\!-\!1/r)2m}+C$ 
   but $N_F(G)<(1-\eta)\Ksharp Cm^{(f-1)/2}$,   
  so $N_F(G)=o(m^{f/2})$.};

\node[box, below=5mm of A] (B) {%
  \textbf{Step 1. Gap-preserving regularization.}
  Deleting $\alpha$-light edges and $\beta$-deficient vertices, 
 the resulting graph $H$ satisfies $e(H):=h>(1-\delta)m$ and
  $\lambda(H)\ge\sqrt{(1-1/r)2h}+(1-\theta)C$.};

\node[box, below=5mm of B] (C) {%
  \textbf{Step 2. Stability \& Refinement.}
  Spectral stability gives a
  near-balanced partition $V(H)\!=\!V_1\cup\cdots\cup V_r$;
  the refinement empties exceptional sets, so the
  Perron vector is near-uniform,
  $x_v^2\le(1+o(1))/{n}$.};

\draw[->] (A)--(B);  \draw[->] (B)--(C);

\node[box, right=of A,anchor=north,
        at={($(A.north)+(68mm,0)$)}] (D) {%
  \textbf{Step 3. Sharp local count.}
  Each class-edge creates at least $(1-o(1))c(n,F)$
  copies of $F$ having it as their only class-edge;
  distinct class-edges give disjoint families.};

\node[box, below=5mm of D] (E) {%
  \textbf{Step 4. Gap conversion.}
  Deleting all $p$ class-edges leaves an $r$-partite
  graph $H^*$, so $\lambda (H^*)\le\sqrt{(1\!-\!1/r)2(h\!-\!p)}$. Comparing with $\lambda(H)$ gives $p\ge(1 \!-\!o(1))Cn$.};

\node[box, below=5mm of E] (F) {%
\textbf{Conclusion.} 
  Summing over all class-edges yields 
  $N_F(G)\!\ge \!(1 \!- \!o(1))\Ksharp Cm^{(f-1)/2}$,
  contradicting the assumption.
 Adding a matching to the Tur\'{a}n graph shows 
  that $\Ksharp$ is best possible.};

\draw[->] (D)--(E);  \draw[->] (E)--(F);

  \draw[->] (C.east) -| ($(C.east)!0.5!(D.west)$) |- (D.west);
\end{tikzpicture}
\caption{Proof outline of Theorem~\ref{thm:linear-small-gap}.}
\label{fig:proof-flow-thm-1-4}
\end{figure}

\subsection{Regularization with almost the same gap}

\label{sec:regularization}

Fix $\alpha \le \frac{1}{2\sqrt{2}}(\frac{1}{2} - \eps_0)$ and $\beta :=2 \eps_2$, so that Lemma \ref{lem:admissible}  applies. 
We  construct a sequence of graphs $ G_1 \supset G_2 \supset\cdots\supset
  G_{\ell}$ 
as follows. Put $G_1=G$. Given $G_{i-1}$, if $G_{i-1}$ contains an $\alpha$-light edge $uv$, or contains a $\beta$-deficient vertex $u_0$, 
then construct $G_i$ from $G_{i-1}$ 
by deleting the edge $uv$, or deleting the vertex $u_0$  together with its incident edges. 
In either case, we discard any isolated vertices, since this does not change $m$ or $\lambda(G)$ and cannot increase $N_F(G)$.  
In what follows, we prove that this deletion process must terminate at a subgraph, say $G_{\ell}$, and 
the total number of deleted edges is less than 
$\lfloor\delta m\rfloor$. Consequently, $G_{\ell}$ contains neither an $\alpha$-light edge nor a $\beta$-deficient vertex.

Throughout the process, we see that $\Phi$ is non-decreasing by 
Lemma~\ref{lem:admissible}. Since
$\Phi(G)=\lambda(G)/\sqrt m>\sqrt{2(1-1/r)}$, we have
$\Phi(G_i)\ge\sqrt{2(1-1/r)}$; equivalently, 
$\lambda(G_i)^2\ge2(1-\frac1r)e(G_i)$ for every $i$.
Thus, the required hypothesis of
Lemma~\ref{lem:admissible} holds at every deletion step.

\begin{claim} \label{clm:regularization}
The terminated graph $G_{\ell}$ 
satisfies 
\begin{equation}\label{eq:surviving-gap}
  e(G_{\ell})>(1-\delta)m,
  \qquad 
  \lambda(G_{\ell})\ge  \sqrt{\left(1- \tfrac{1}{r}\right) 2e(G_{\ell})}+(1-\theta)C.
\end{equation}
\end{claim}

\begin{proof}
Suppose on the contrary that the process deletes at least $\lfloor\delta m\rfloor$ edges. Let $G_{1}\supset G_{2}\supset\cdots\supset G_{k}$ be the initial segment up to the first index $k$ at which $e(G_{1})-e(G_{k})\ge\lfloor\delta m\rfloor$. Since a single deletion step removes at most $\Delta(G)=O(\sqrt m)$ edges, we also have $e(G_{1})-e(G_{k})\le\lfloor\delta m\rfloor+O(\sqrt m)$.   
 Write
\[  H=G_k,
  \qquad
  h=e(H), \qquad 
  \Delta_i=e(G_i)-e(G_{i+1}).
\]
Since $G_{i+1}$ is obtained from $G_i$ by deleting an $\alpha$-light edge, or deleting a $\beta$-deficient vertex, using Lemma~\ref{lem:admissible}, we have 
\[
  \Phi(G_{i+1})-\Phi(G_i)
  \ge \frac{\eps_0\Delta_i}{e(G_i)}
  \ge \frac{\eps_0\Delta_i}{m}.
\]
Summing over 
the performed deletion steps gives
\[
  \Phi(H)-\Phi(G)\ge \frac{\eps_0(m-h)}m.
\]
The choice of $k$ gives $m-h\ge\lfloor\delta m\rfloor\ge\delta m/2$ for all sufficiently large $m$.  Therefore
\[
  \Phi(H)
  \ge \Phi(G)+\frac{\eps_0\delta}{2}
  \ge \sqrt{2\left(1- \tfrac{1}{r}\right)}+\frac{\eps_0\delta}{3}. 
\]
Set $  \vartheta_0=\big(\sqrt{2(1-1/r)}+\frac{\eps_0\delta}{6}\big)^2-2(1-\frac1r)>0$. 
Then the preceding inequality implies
\[
  \lambda(H)>\sqrt{\left(2\left(1-\tfrac1r \right)+\vartheta_0 \right)h}.
\]
On the other hand, we have $h\ge m-\lfloor\delta m\rfloor-O(\sqrt m)\ge(1-2\delta)m$ and
\[
  N_F(H)\le N_F(G)=O(m^{(f-1)/2})=o(h^{f/2}).
\]
This contradicts the spectral supersaturation in Theorem~\ref{thm:edge-spec-ESS} with parameter $\vartheta_0$.  Thus, the deletion process cannot reach a stage where $\lfloor\delta m\rfloor$ edges have been deleted. So $h>(1-\delta)m$.

Therefore, 
the deletion process stops because $G_{\ell}$ contains no $\alpha$-light edge and no $\beta$-deficient vertex.  
Since $\Phi$ is non-decreasing along the process, 
it follows that 
\[
  \lambda(G_{\ell}) 
  \ge \Phi(G)\sqrt{e(G_{\ell})}
  \ge \Big(\sqrt{2\left( 1- \tfrac{1}{r}\right)}+\tfrac C{\sqrt m}\Big)\sqrt{e(G_{\ell})}.
\]
Because $e(G_{\ell})/m>1-\delta$ and $\delta\ll\theta$,
we get $ C\sqrt{ e(G_{\ell}) /m}\ge (1-\theta)C$. 
This proves \eqref{eq:surviving-gap}. 
\end{proof}

\subsection{Stability and structural refinement}

The structural refinement in this
subsection follows the same approach of Fang,
Lin and Zhai~\cite{FLZ}; we reproduce it,
with minor changes, so that the paper is
self-contained.

Let $H:=G_{\ell}$ be the terminated graph obtained in Section \ref{sec:regularization}.  
Denote $h=e(H)$ and $  n=|V(H)|$. 
By Claim~\ref{clm:regularization}, 
we get 
$  h\ge(1-\delta)m$ 
and
$ \lambda(H)
  \ge \sqrt{(1- {1}/{r}) 2h}+(1-\theta)C.$ 
Also, we have $$ N_F(H)\le N_F(G)< (1- \eta)\kappa_F C m^{(f-1)/2}.$$ 
The hypothesis states $C\le C_0$ for some constant $C_0$; this implies the explicit bound
\[
  N_F(H)
  \le
 (1- \eta)\kappa_F C_0(1-\delta)^{-(f-1)/2}h^{(f-1)/2} = o(h^{f/2}).
\]
Therefore, the supersaturation-stability in  Theorem~\ref{thm:edge-spectral-stability} yields a Tur\'an graph $T_{n',r}$ such that 
\[   V(T_{n',r})\subseteq V(H) 
\quad \text{and} \quad d(H,T_{n',r})\le \eps_0 h. \]
Let $\bm{x}=(x_v)_{v\in V(H)}$  be the unit Perron vector of $H$. 
 The regularization process in Section \ref{sec:regularization} terminates at $H$, so every edge $uv$ satisfies $x_ux_v>\alpha/\sqrt h$ and $H$ has no isolated vertices.   
 
Since  $  \lambda^2(H) \ge \frac43h$, 
Lemma~\ref{lem:LLZ} implies 
$\max_{v\in V(H)} x_v =O(h^{-1/4})$. 
Combining with  
$x_ux_v > \alpha / \sqrt{h}$, we have $\min_{u\in V(H)} x_u = \Omega (h^{-1/4})$, which together with $\sum_{u\in V(H)} x_u^2 =1$ yields $ n=O(\sqrt h)$. Combining with $n\ge \lambda (H) \ge \sqrt{h}$, we obtain 
$n= \Theta (\sqrt{h}). $  
Moreover, we have $ \lambda (H) x_u = \sum_{v\in N_H(u)} x_v  \le d_H(u) \cdot O(h^{-1/4})$. Since $\lambda (H) \ge \sqrt{h}$ and $x_u = \Omega(h^{-1/4})$, we get 
$ d_H(u) =\Omega (\sqrt{h})$ for all $u\in V(H)$. 
In particular, $\delta(H)\ge c_r \sqrt h$ for some constant $c_r>0$.

\begin{claim}\label{clm:n-size}
Let $ \psi=\sqrt{\frac{2h}{1-1/r}}$. Then 
$ (1-\eps_1)\psi\le n\le (1+\eps_1)\psi.$
\end{claim}

\begin{proof}
 Since  $|e(T_{n',r})-h|\le d(H,T_{n',r})\le\eps_0h$,
we have $(1-\eps_0)h\le e(T_{n',r})\le(1+\eps_0)h$. Writing $n'= \ell r+s$ with $0\le s<r$ gives
$e(T_{n',r})=\frac{r-1}{2r}(n')^2-\frac{s(r-s)}{2r}$, and hence
$$\frac{r-1}{2r}(n')^2-\frac r8\le e(T_{n',r})\le\frac{r-1}{2r}(n')^2.$$ 
Combining these two pairs
of bounds,
\[
  (1-\eps_0)\psi^2\le (n')^2\le (1+\eps_0)\psi^2+O_r(1).
\]
As $\psi\to\infty$, the additive $O_r(1)$ is at most $\eps_0\psi^2$ for all large $h$, so
$$(1-\eps_0)\psi^2\le (n')^2\le(1+2\eps_0)\psi^2.$$
 By $\eps_0\ll\eps_1^2$, this yields
\[
  (1-\eps_1^2/4)\psi\le n'\le (1+\eps_1^2/4)\psi.
\]
Let $R=V(H)\setminus V(T_{n',r})$. Every edge incident with $R$ is absent from $T_{n',r}$
and is counted at most twice in $\sum_{v\in R}d_H(v)$, so
$\eps_0h\ge d(H,T_{n',r})\ge\frac12\sum_{v\in R}d_H(v)\ge\frac12|R|\delta(H)$. With
$\delta(H)\ge c_r \sqrt h$ this gives $|R|\le 2\eps_0\sqrt h/c_r $; since
$n'=\Theta_r(\sqrt h)$ and $\eps_0\ll\eps_1^2$, we obtain $|R|\le \frac{1}{4}{\eps_1^2}n'$.
Therefore
\[
  (1-\eps_1^2)\psi\le n'\le n=n'+|R|\le 
  \big(1+ \tfrac{1}{4}{\eps_1^2} \big)n'\le (1+\eps_1^2)\psi,
\]
which proves the claim.
\end{proof}

Choose a partition
$  V(H)=V_1\cup\cdots\cup V_r$ 
that maximizes $\sum_{i< j} e(V_i,V_j)$.  
The edges in $\bigcup_{i=1}^r H[V_i]$ are called class-edges. The edges between $V_i$ and $V_j$ for some $i\neq j$ are called cross-edges.

\begin{claim}\label{clm:balanced-partition}
We have $  e(H)\ge \frac{r-1}{2r}n^2-3\eps_1^2n^2$, and 
the partition satisfies
\[
  \sum_{i=1}^r e(H[V_i])\le \eps_1^2n^2, \quad \text{and~} \quad 
    \left||V_i|-\frac nr\right|\le 3\eps_1n
  \quad
  \text{for every }i\in[r]. 
\]
\end{claim}

\begin{proof}
Let $U_1,\ldots,U_r$ be the color classes of $T_{n',r}$ and $R=V(H)\setminus V(T_{n',r})$,
and set $U_1^\ast:=U_1\cup R$ and $U_i^\ast:=U_i$ for $2\le i\le r$. Since each $U_i$ is
independent in $T_{n',r}$ and the vertices of $R$ lie outside $T_{n',r}$, every edge counted
by $\sum_i e(H[U_i^\ast])$ belongs to $E(H)\setminus E(T_{n',r})$; in particular the edges
inside $R$ and between $R$ and $U_1$ all fall in $e(H[U_1^\ast])$. As the sets $U_i^\ast$
partition $V(H)$, these within-part edges are distinct, so
$\sum_{i=1}^r e(H[U_i^\ast])\le d(H,T_{n',r})\le\eps_0h$.
Because the partition $V_1,\ldots,V_r$ maximizes the number of cross-edges, it minimizes the
number of class-edges. Using $h\le\binom n2$ and $\eps_0\ll\eps_1^2$, 
we get 
\[
  \sum_{i=1}^r e(H[V_i])\le\sum_{i=1}^r e(H[U_i^\ast])\le\eps_0h\le\eps_1^2n^2.
\]

We now prove the balance estimate. Put $\rho=\max_{i}\big||V_i|-\tfrac nr\big|$ and assume
$\rho=\big||V_1|-\tfrac nr\big|$. By the Cauchy--Schwarz inequality
$\big(\sum_{i\ge2}|V_i|\big)^2\le(r-1)\sum_{i\ge2}|V_i|^2$, so
$$\sum_{2\le i<j\le r}|V_i||V_j|\le\frac{r-2}{2(r-1)}(n-|V_1|)^2.$$
 Hence
\[
\begin{aligned}
  e(H)
  &\le\sum_{1\le i<j\le r}|V_i||V_j|+\sum_{i=1}^r e(H[V_i])\\
  &\le|V_1|(n-|V_1|)+\frac{r-2}{2(r-1)}(n-|V_1|)^2+\eps_1^2n^2 \\
    &=\frac{r-1}{2r}n^2-\frac{r}{2(r-1)}\rho^2+\eps_1^2n^2,
\end{aligned}
\]
the last equality being the exact identity for $\rho=\big||V_1|-\tfrac nr\big|$.

For a matching lower bound, recall from Claim~\ref{clm:n-size} that
$n-n'=|R|\le\frac{1}{4}\eps_1^2 n'\le\frac{1}{2} \eps_1^2 n$. Building $T_{n,r}$ from $T_{n',r}$ one
vertex at a time adds at most the current vertex count per step, so
$e(T_{n,r})-e(T_{n',r})\le\sum_{t=n'}^{n-1}t\le(n-n')n\le\frac{1}{2}\eps_1^2n^2$. 
Together with
$e(T_{n,r})\ge\frac{r-1}{2r}n^2-\frac r8$ and $\eps_0h\le\eps_1^2n^2$, this gives
\[
  e(H)\ge e(T_{n',r})-\eps_0h\ge\frac{r-1}{2r}n^2-3\eps_1^2n^2.
\]
Comparing the two bounds yields $\frac{r}{2(r-1)}\rho^2\le4\eps_1^2n^2$, hence $\rho<3\eps_1n$.
\end{proof}

In the sequel, we define the exceptional sets 
$S^{(1)},S^{(2)}$ and $W_i$ as follows. 
\begin{itemize}
\item 
For $k\in\{1,2\}$, let $  S^{(k)}=
  \{
  v\in V(H):
  d_H(v)\le \left(1-\frac1r-4\eps_k\right)n \}$; 
  
  \item 
For each $i\in[r]$, we define $ W_i=
  \left\{
  v\in V_i:
  d_{V_i}(v)\ge 4\eps_1n
  \right\}$; 
  
  \item 
  For $k\in\{1,2\}$, we denote 
$  V_i^{(k)}=V_i\setminus(W_i\cup S^{(k)})$. 
  \end{itemize}
Since $\eps_1<\eps_2$, we have $S^{(2)}\subseteq S^{(1)}$ and $V_i^{(1)}\subseteq V_i^{(2)}$. 
The set $S^{(2)}$ keeps
the set $V_i^{(2)}$ large for the counting arguments, while the set $S^{(1)}$ isolates
the nearly-maximal-degree vertices on which the Perron vector is almost constant; the gap
$\eps_1\ll\eps_2$ forces the $\eps_2$-scale degree deficit in Claim~\ref{clm:S-empty}. 

\begin{claim}\label{clm:S-small}
For each $k=1,2$, we have  $  |S^{(k)}|\le \eps_1n.$
\end{claim}

\begin{proof}
Suppose, to the contrary, that $|S^{(k)}|>\eps_1n$ for some $k\in\{1,2\}$. Pick
$S'\subseteq S^{(k)}$ with $|S'|=s=\lfloor\eps_1n\rfloor$, let $J=H[V(H)\setminus S']$, and write
$N=|J|=n-s$. Every $v\in S'$ satisfies $d_H(v)\le(1-\frac1r-4\eps_1)n$: this is the definition
when $k=1$, and follows from $S^{(2)}\subseteq S^{(1)}$ when $k=2$.

By Claim~\ref{clm:balanced-partition}, $e(H)\ge\frac{r-1}{2r}n^2-3\eps_1^2n^2$. Deleting $S'$
removes at most $\sum_{v\in S'}d_H(v)$ edges (an edge inside $S'$ is counted twice in this sum,
which only makes it larger, so the bound stays valid), so
$e(J)\ge e(H)-s(1-\frac1r-4\eps_1)n$. Using $e(T_{N,r})\le\frac{r-1}{2r}(n-s)^2$ and
$\eps_1n-1\le s\le\eps_1n$,
\[
\begin{aligned}
  e(J)-e(T_{N,r})
  \ge -3\eps_1^2n^2+4\eps_1sn-\tfrac{r-1}{2r}s^2 
  \ge \tfrac{r+1}{2r}\eps_1^2n^2-4\eps_1n
  \ \ge\ \tfrac14\eps_1^2n^2
\end{aligned}
\]
for all large $n$, since $r\ge3$. As $N\le n$, this gives $e(J)\ge e(T_{N,r})+\frac18\eps_1^2N^2$.

Applying Lemma~\ref{lem:ES-supersaturation} with parameter $\frac18\eps_1^2$ gives a constant
$b=b(F,r,\eps_1)>0$ with $N_F(J)\ge bN^f$. Since $N\ge(1-\eps_1)n$ and $h=\Theta(n^2)$, we get
$N_F(H)\ge N_F(J)\ge b'h^{f/2}$ for some $b'=b'(F,r,\eps_1)>0$. On the other hand,
$N_F(H)\le  N_F (G)<(1-\eta)\kappa_FCm^{(f-1)/2}$, and $h\ge(1-\delta)m$ gives
$ N_F (H)=O(h^{(f-1)/2})=o(h^{f/2})$, a contradiction. Hence
$|S^{(k)}|\le\eps_1n$ for each $k=1,2$.
\end{proof}

\begin{claim}\label{clm:W-small}
We have
\[
  \sum_{i=1}^r |W_i|\le \eps_1n.
\]
\end{claim}

\begin{proof}
For each $i$, we have 
\[
  e(H[V_i])
  =
  \frac12\sum_{v\in V_i}d_{V_i}(v)
  \ge
  \frac12\sum_{v\in W_i}d_{V_i}(v)
  \ge
  2\eps_1n|W_i|.
\]
By Claim~\ref{clm:balanced-partition}, we get 
$2\eps_1n\sum_{i=1}^r |W_i| \le 
  \sum_{i=1}^r e(H[V_i])\le \eps_1^2n^2$. 
So $ \sum_{i=1}^r |W_i|
  \le
  \frac12\eps_1n
  \le
  \eps_1n.$ 
\end{proof}

We use a common-neighborhood estimate.  For a set $A\subseteq V(H)$ and a vertex $u$, let
\[
  N_A(u)=N_H(u)\cap A.
\]

\begin{claim}\label{clm:common-neighborhood}
Fix $k\in\{1,2\}$, $i_0\in[r]$, and $0\le t\le f$. 
\begin{enumerate}[label=\textup{(\roman*)}]
\item If $u_1,\ldots,u_t\in\bigcup_{i\ne i_0}V_i^{(k)}$, then
  $\big|\bigcap_{j=1}^t N_{V_{i_0}^{(k)}}(u_j)\big|\ge\big(\tfrac1r-4f^2\eps_2\big)n$.
\item If in addition $u_0\in\bigcup_{i\ne i_0}(W_i\setminus S^{(k)})$, then
  $\big|N_{V_{i_0}^{(k)}}(u_0)\cap\bigcap_{j=1}^t N_{V_{i_0}^{(k)}}(u_j)\big|\ge\tfrac{n}{2r^2}$.
\end{enumerate}
\end{claim}

\begin{proof}
Throughout we use, from Claim~\ref{clm:balanced-partition}, that
$(\frac1r-3\eps_1)n\le|V_i|\le(\frac1r+3\eps_1)n$ for all $i$, and from
Claims~\ref{clm:S-small} and~\ref{clm:W-small} that $|W_{i_0}\cup S^{(k)}|\le2\eps_1n$;
since $\eps_1\ll\eps_2$, both $3\eps_1$ and $2\eps_1$ are below $2\eps_2$.

(i) For $t=0$, we have $|V_{i_0}^{(k)}|\ge(\frac1r-3\eps_1)n-2\eps_1n\ge(\frac1r-4f^2\eps_2)n$.
Let $1\le t\le f$ and write $u_j\in V_{i_j}^{(k)}$ with $i_j\ne i_0$. As
$u_j\notin W_{i_j}\cup S^{(k)}$, we have $d_{V_{i_j}}(u_j)<4\eps_1n<2\eps_2n$ and
$d_H(u_j)>(1-\frac1r-4\eps_2)n$. Using $|V_i|<(\frac1r+2\eps_2)n$ and $f\ge r+1$, we have 
\[
  d_{V_{i_0}}(u_j)=d_H(u_j)-d_{V_{i_j}}(u_j)-\sum_{i\notin\{i_0,i_j\}}d_{V_i}(u_j)
>\Big(\tfrac1r-2(r+1)\eps_2\Big)n\ge\Big(\tfrac1r-2f\eps_2\Big)n.
\]
By the bound $|A_1\cap\cdots\cap A_t|\ge\sum_j|A_j|-(t-1)|\Omega|$ for subsets of $\Omega=V_{i_0}$, we get 
\[
  \Big|\bigcap_{j=1}^t N_{V_{i_0}}(u_j)\Big|
  >t\Big(\tfrac1r-2f\eps_2\Big)n-(t-1)\Big(\tfrac1r+2\eps_2\Big)n
  \ge\Big(\tfrac1r-3f^2\eps_2\Big)n.
\]
Removing the at most $2\eps_1n\le2\eps_2n$ vertices of $W_{i_0}\cup S^{(k)}$ restricts the
intersection to $V_{i_0}^{(k)}$ at a cost of $2\eps_2n$, and $3f^2+2\le4f^2$; this gives~(i).

(ii) Let $u_0\in W_{i'}\setminus S^{(k)}$ with $i'\ne i_0$, and set
$B=\bigcap_{j=1}^t N_{V_{i_0}^{(k)}}(u_j)$ (so $B=V_{i_0}^{(k)}$ when $t=0$); by (i),
$|B|\ge(\frac1r-4f^2\eps_2)n$. Since the partition maximizes the number of cross-edges, moving
$u_0$ to any other part cannot increase that number, so $d_{V_{i'}}(u_0)\le d_{V_j}(u_0)$ for all
$j$; summing gives $d_H(u_0)\ge r\,d_{V_{i'}}(u_0)$, i.e.\ $d_{V_{i'}}(u_0)\le\frac1r d_H(u_0)$.
With $d_H(u_0)>(1-\frac1r-4\eps_2)n$ and $|V_i|<(\frac1r+2\eps_2)n$,
\[
  d_{V_{i_0}}(u_0)\ge\Big(1-\tfrac1r\Big)d_H(u_0)-(r-2)\Big(\tfrac1r+2\eps_2\Big)n
  >\Big(\tfrac1{r^2}-2f\eps_2\Big)n,
\]
using $f\ge r+1$. Since $B\subseteq V_{i_0}^{(k)}$ we have
$N_{V_{i_0}^{(k)}}(u_0)\cap B=N_{V_{i_0}}(u_0)\cap B$. Then 
\[
  \Big|N_{V_{i_0}^{(k)}}(u_0)\cap B\Big|\ge d_{V_{i_0}}(u_0)+|B|-|V_{i_0}|
  \ge\Big(\tfrac1{r^2}-(4f^2+2f+2)\eps_2\Big)n\ge 
  \tfrac{n}{2r^2},
\]
the last step by the choice of $\eps_2$. This proves (ii).
\end{proof}

For a class-edge $e$ of $H$, let $N_F(H, e)$ be the number of
copies of $F$ in $H$ whose image-edge set contains exactly one
class-edge, namely $e$. The families of copies counted by $N_F(H, e)$, over all distinct $e$, are pairwise disjoint, since each copy has exactly one class-edge in its image-edge set.

\begin{claim}\label{clm:local-count}
There is a constant $\gamma=\gamma(F,r)>0$ such that
$N_F(H,uv)\ge\gamma\,c(n,F)$ whenever $uv\in E(H[V_i])$
with $u\in V_i\setminus S^{(2)}$ and $v\in V_i^{(2)}$.
\end{claim}

\begin{proof}
We build many copies of $F$ by mapping a fixed template
into the parts $V_1,\dots,V_r$, sending the unique
within-class edge to $uv$ and every other edge to a
cross-edge.

\emph{Template.} By symmetry take $i=1$. Since $F$ is
color-critical and deleting an edge drops $\chi$ by at most
one, it has an edge $xy$ with $\chi(F-xy)=r$. Fix a proper
$r$-coloring of $F-xy$; it uses all $r$ colors, and $x,y$
get the same color, say color $1$ (otherwise the coloring
would stay proper after restoring $xy$, contradicting
$\chi(F)=r+1$). Let $L_1,\dots,L_r$ be the color classes,
with $x,y\in L_1$. Then $xy$ is the only edge of $F$ inside
a class, so mapping each $L_j$ into $V_j$ and $xy$ to a
class-edge of $H[V_1]$ turns every other edge of $F$ into a
cross-edge.

\emph{Choosing $L_1^{*}$.} We count copies with
$x\mapsto u$ and $y\mapsto v$ (hence $xy\mapsto uv$) whose
only class-edge is $uv$; counting only this assignment
already gives a lower bound. Pick the remaining $|L_1|-2$
vertices of $L_1$ inside $V_1^{(2)}\setminus\{u,v\}$. By
Claims~\ref{clm:balanced-partition},~\ref{clm:S-small}
and~\ref{clm:W-small}, we have
$|V_1^{(2)}|\ge(\tfrac1r-3\eps_1)n-2\eps_1n\ge\tfrac n{2r}$,
so $|V_1^{(2)}\setminus\{u,v\}|\ge\lfloor n/(3r)\rfloor$
for large $n$, giving at least
$\binom{\lfloor n/(3r)\rfloor}{|L_1|-2}$ choices (read as
$1$ when $|L_1|=2$). Write $L_1^{*}$ for the resulting set,
which contains $u$ and $v$.

\emph{Greedy extension.} Build $L_2^{*},\dots,L_r^{*}$ one
class at a time. Suppose $L_1^{*},\dots,L_{j-1}^{*}$ are
chosen; they have at most $f$ vertices, and none lies in
$V_j^{(2)}$ since they sit in $V_1,\dots,V_{j-1}$. If
$u\in V_1^{(2)}$, then all chosen vertices lie in their own
$V_i^{(2)}$ and Claim~\ref{clm:common-neighborhood}(i)
applies; if $u\in W_1\setminus S^{(2)}$, then
Claim~\ref{clm:common-neighborhood}(ii) applies with
$u_0=u$ and the other chosen vertices as $u_1,\dots,u_t$.
Either way, by the choice of $\eps_2$, these vertices have
at least $\tfrac n{2r^2}$ common neighbors in $V_j^{(2)}$,
each adjacent to every earlier vertex; so
$L_j^{*}\subseteq V_j^{(2)}$ can be chosen in at least
$\binom{\lfloor n/(2r^2)\rfloor}{|L_j|}$ ways.

\emph{Each choice is a copy.} The complete $r$-partite
graph on $L_1^{*},\dots,L_r^{*}$ together with $uv$
contains a copy of $F$ with $xy\mapsto uv$: the only
within-class edge $xy$ maps to $uv$, while every edge of
$F-xy$ joins two classes and the greedy choice makes all
such cross-pairs adjacent (extra edges of $H$ inside the
sets are irrelevant, since the copy is non-induced).
Distinct choices of the $L_j^{*}$ give distinct image
sets, recovered by intersecting the image with the parts
$V_j$; fixing one bijection $L_j\to L_j^{*}$ for each $j$
(with $x\mapsto u$, $y\mapsto v$) turns each choice into an
edge-preserving injection, and passing to unlabeled copies
merges at most $|\Aut(F)|$ of them.

\emph{Counting.} Since
$\binom{\lfloor An\rfloor}{a}\ge c(A,a)\,n^a$ for fixed
$a\ge0$, $A>0$ and large $n$, the product of the binomial
factors is at least a positive constant (depending only on
$F$ and $r$) times
$n^{|L_1|-2}\prod_{j\ge2}n^{|L_j|}=n^{f-2}$. Hence
$N_F(H,uv)\ge b_F\,n^{f-2}$ for some $b_F>0$. Since
$c(n,F)<2\alpha_Fn^{f-2}$ for large $n$ by
Lemma~\ref{lem:Mub2010-c}, we get
$N_F(H,uv)\ge b_F n^{f-2}\ge\tfrac{b_F}{2\alpha_F}\,c(n,F)$,
and the claim holds with $\gamma=b_F/(2\alpha_F)$.
\end{proof}

\begin{claim} \label{cl-W-sub-S2}
We have $\bigcup_{i=1}^r W_i\subseteq S^{(2)}$. 
\end{claim}

\begin{proof}
Suppose on the contrary that 
$\bigcup_{i=1}^r W_i\not\subseteq S^{(2)}$.  
Choose $u\in W_i\setminus S^{(2)}$.  Then $d_{V_i}(u)\ge4\eps_1n$.  By Claims~\ref{clm:S-small} and~\ref{clm:W-small}, we have 
$
  |W_i\cup S^{(2)}|\le2\eps_1n$.
Hence $u$ has at least $2\eps_1n$ neighbors $v$ in $V_i^{(2)}$.  For each such $v$, Claim~\ref{clm:local-count} gives
$$  N_F(H,uv)\ge \gamma c(n,F).$$ 
The families counted for different edges $uv$ are disjoint because each counted embedding uses exactly one class-edge of the partition, namely $uv$.  Therefore, using Lemma~\ref{lem:Mub2010-c}, we get 
\[
   N_F (H)
  \ge 2\eps_1n \cdot \gamma c(n,F)
  \ge \eps_1\gamma\alpha_F n^{f-1}
\]
for all sufficiently large $n$.  By Claim~\ref{clm:n-size} and $h>(1-\delta)m$, 
we have 
\[
  n^{f-1}
  \ge
  (1-O_F(\eps_1+\delta))
  \left(\tfrac{2mr}{r-1}\right)^{(f-1)/2}.
\]
Thus
\[
   N_F (G)\ge N_F (H)
  \ge
  (1-O_F(\eps_1+\delta))\eps_1\gamma\alpha_F
  \left(\tfrac{2r}{r-1}\right)^{(f-1)/2}m^{(f-1)/2}.
\]
Since $C\le C_0\le\eps_1\gamma/100$, this is larger than the right-hand side of \eqref{eq:sharp-counter} for all sufficiently large $m$, which leads to a contradiction with the assumption. Thus, we must have 
$\bigcup_{i=1}^r W_i\subseteq S^{(2)}$. 
\end{proof}

Since $S^{(2)}\subseteq S^{(1)}$, we have $\bigcup_iW_i\subseteq S^{(1)}$. 
Let $\bm{x}=(x_v)_{v\in V(H)}$ be the unit Perron eigenvector of $H$, and 
let $u^*\in V(H)$ be a vertex such that $x_{u^*}=\max_{v\in V(H)} x_v$. 

\begin{claim}\label{clm:perron-almost-constant}
Every $u\in\bigcup_{i=1}^r V_i^{(1)}$ satisfies
$x_u\ge(1-16f^2\eps_1)x_{u^\ast}$.
\end{claim}

\begin{proof}
Throughout we use $|V_j|\le(\tfrac1r+3\eps_1)n$
(Claim~\ref{clm:balanced-partition}). By Claim~\ref{clm:n-size},
$\lambda(H)\ge\sqrt{2(1-\tfrac1r)h}=(1-\tfrac1r)\psi$ with
$\psi\ge(1-\eps_1)n$, so
$\lambda(H)\ge(1-\tfrac1r-\eps_1)n$; since $r\ge3$, this
gives $\lambda(H)>\tfrac n2$.

\emph{Non-neighbors.} Fix
$u\in V_i^{(1)}=V_i\setminus(W_i\cup S^{(1)})$, so
$d_{V_i}(u)<4\eps_1n$ and $d_H(u)>(1-\tfrac1r-4\eps_1)n$.
For $j\ne i$, splitting off the parts $V_i$ and $V_s$
($s\ne i,j$) gives
$d_{V_j}(u)=d_H(u)-d_{V_i}(u)-\sum_{s\ne i,j}d_{V_s}(u)
>(\tfrac1r-3f\eps_1)n$ (using $f\ge r+1$). 
Hence
$
  \sum_{j\ne i}\bigl(|V_j|-d_{V_j}(u)\bigr)
  <3(f+1)(r-1)\eps_1n<4(f^2-1)\eps_1n$.

\emph{Same part.} Let $u_1,u_2\in V_i^{(1)}$ with
$x_{u_1}\ge x_{u_2}$. From the eigenvalue equation, we have 
$$\lambda(H)(x_{u_1}-x_{u_2})
=\sum_{v\in V(H)}(a_{u_1v}-a_{u_2v})x_v
\le\sum_{v\in N_H(u_1)\setminus N_H(u_2)}x_v.$$
Each such $v$ lies either in $N_{V_i}(u_1)$ (at most
$d_{V_i}(u_1)<4\eps_1n$ vertices) or, for some $j\ne i$, in
$V_j\setminus N_{V_j}(u_2)$, and the latter total is at
most $4(f^2-1)\eps_1n$ by the non-neighbor bound applied to
$u_2$. As each $x_v\le x_{u^\ast}$, we get
$\lambda(H)(x_{u_1}-x_{u_2})\le4f^2\eps_1n\,x_{u^\ast}$, and
$\lambda(H)>\tfrac n2$ then yields
$x_{u_1}-x_{u_2}\le8f^2\eps_1x_{u^\ast}$. Thus, the Perron
coordinates inside one $V_i^{(1)}$ differ by at most
$8f^2\eps_1x_{u^\ast}$.

\emph{Location of $u^\ast$.} Since
$\lambda(H)x_{u^\ast}=\sum_{v\in N_H(u^\ast)}x_v
\le d_H(u^\ast)x_{u^\ast}$, we have
$d_H(u^\ast)\ge\lambda(H)\ge(1-\tfrac1r-\eps_1)n$, so
$u^\ast\notin S^{(1)}$. As $\bigcup_iW_i\subseteq S^{(1)}$,
this places $u^\ast\in\bigcup_iV_i^{(1)}$, say
$u^\ast\in V_1^{(1)}$. 
Applying the same-part bound established above,  
we then obtain 
$x_u\ge(1-8f^2\eps_1)x_{u^\ast}$ for all $u\in V_1^{(1)}$.

\emph{Other parts.} Suppose some
$u_0\in\bigcup_{i\ge2}V_i^{(1)}$, say $u_0\in V_2^{(1)}$,
had $x_{u_0}<(1-16f^2\eps_1)x_{u^\ast}$. Since coordinates
in $V_2^{(1)}$ differ by at most $8f^2\eps_1x_{u^\ast}$,
every $u\in V_2^{(1)}$ then satisfies
$x_u<(1-8f^2\eps_1)x_{u^\ast}\le(1-8r^2\eps_1)x_{u^\ast}$,
the last step using $f\ge r$. We now bound
$\lambda(H)x_{u^\ast}=\sum_i\sum_{v\in N_{V_i}(u^\ast)}x_v$
part by part. For $V_2$, the at most
$|V_2|\le(\tfrac1r+3\eps_1)n$ neighbors in $V_2^{(1)}$ each
contribute at most $(1-8r^2\eps_1)x_{u^\ast}$, while the at most
$|W_2\cup S^{(1)}|\le2\eps_1n$ neighbors outside contribute
at most $ x_{u^\ast}$, so
$\sum_{v\in N_{V_2}(u^\ast)}x_v
\le \left((1-8r^2\eps_1)(\tfrac1r+3\eps_1)n+2\eps_1n \right) x_{u^\ast}$. For $V_1$, since
$u^\ast\notin W_1$ we have $d_{V_1}(u^\ast)<4\eps_1n$, so
this part contributes at most $4\eps_1n\,x_{u^\ast}$; each
$V_i$ with $i\ge3$ contributes at most
$(\tfrac1r+3\eps_1)n\,x_{u^\ast}$. 
Adding these, it follows that 
\[
  \lambda(H)
  \le(r-1-8r^2\eps_1)\bigl(\tfrac1r+3\eps_1\bigr)n+6\eps_1n
  \le\bigl(1-\tfrac1r-4r\eps_1\bigr)n,
\]
where the last step expands the coefficient to
$1-\tfrac1r-(5r-3)\eps_1-24r^2\eps_1^2$ and uses $r\ge3$.
This contradicts $\lambda(H)\ge(1-\tfrac1r-\eps_1)n$, so no
such $u_0$ exists, and the claim follows.
\end{proof} 

\begin{claim}\label{clm:S-empty}
We have $S^{(2)}=\varnothing$.
\end{claim}

\begin{proof}
Suppose not, and fix $u\in S^{(2)}$; we show $u$ is $\beta$-deficient with $\beta=2\eps_2$,
contradicting Claim~\ref{clm:regularization}. Write $\psi=\sqrt{2h/(1-1/r)}$, so $2h=(1-\frac1r)\psi^2$ and, by
Claim~\ref{clm:n-size}, $(1-\eps_1)\psi\le n\le(1+\eps_1)\psi$.

Note that $\bigcup_iW_i\subseteq S^{(2)}\subseteq S^{(1)}$, so $V_i^{(1)}=V_i\setminus S^{(1)}$ and,
by Claim~\ref{clm:S-small}, $\sum_i|V_i^{(1)}|=n-|S^{(1)}|\ge(1-\eps_1)n\ge(1-\eps_1)^2\psi$.
Claim~\ref{clm:perron-almost-constant} then gives
\[
  1=\sum_{v\in V(H)}x_v^2\ge\sum_{i=1}^r\sum_{v\in V_i^{(1)}}x_v^2
  \ge(1-\eps_1)^2(1-16f^2\eps_1)^2\,\psi\,x_{u^\ast}^2\ge\Big(1-\tfrac12\eps_2\Big)\psi\,x_{u^\ast}^2,
\]
the last step using $\eps_1\ll\eps_2$. Hence $\psi x_{u^\ast}^2\le(1-\tfrac12\eps_2)^{-1}\le1+\eps_2$.

The fixed Perron vector is positive on $V(H)$ (otherwise some edge would be light, against
Claim~\ref{clm:regularization}). Using $\lambda(H)\ge\sqrt{2(1-\frac1r)h}=(1-\frac1r)\psi$ and
$x_{u^\ast}^2\le(1+\eps_2)/\psi$, we have 
\[
  (1-\tfrac1r)\psi\,x_u\le\lambda(H)x_u=\sum_{v\in N_H(u)}x_v\le d_H(u)x_{u^\ast}
  \le d_H(u)\sqrt{\tfrac{1+\eps_2}{\psi}},
\]
so $x_u^2\le d_H(u)^2(1+\eps_2)/((1-\frac1r)^2\psi^3)$. As $u\in S^{(2)}$ and $n\le(1+\eps_1)\psi$, 
\[
  d_H(u)\le\left(1-\tfrac1r-4\eps_2\right) 
  (1+\eps_1)\psi 
  \le \left(1-\tfrac1r-3\eps_2 \right) \psi\le\psi,
\]
using $\eps_1\ll\eps_2$ and $1-\frac1r<1$. Combining these with $2h=(1-\frac1r)\psi^2$, we get 
\[
  \frac{2hx_u^2}{d_H(u)}\le\frac{d_H(u)(1+\eps_2)}{(1-1/r)\psi}
  \le\frac{(1-\frac1r-3\eps_2)(1+\eps_2)}{1-1/r}
  =1+\eps_2-\frac{3\eps_2(1+\eps_2)}{1-1/r}\le1-2\eps_2,
\]
where the last step uses $1-\frac1r\le1$. Thus $2hx_u^2\le(1-2\eps_2)d_H(u)$. Moreover
$x_u^2\le x_{u^\ast}^2\le(1+\eps_2)/\psi\le\eps_2$ for large $h$ (as $\psi\to\infty$), and
$d_H(u)\le\psi$ was just shown. With $\beta=2\eps_2$, all three conditions of
Definition~\ref{def:beta-deficient} hold, so $u$ is $\beta$-deficient, a contradiction. Hence
$S^{(2)}=\varnothing$.
\end{proof}

\subsection{Local counts from each class-edge}

Claim \ref{cl-W-sub-S2} gives  $\bigcup_{i=1}^r W_i\subseteq S^{(2)}$. Then Claim~\ref{clm:S-empty}  implies $ W_i=\varnothing$ for all $i\in [r]$.  

\begin{claim} \label{clm:clean-case}
The following estimates hold.
\begin{enumerate}[label=\textup{(\roman*)}]
\item Every vertex has internal degree at most $4\eps_1n$ with respect to the partition $V_1\cup\cdots\cup V_r$.
\item Every vertex is incident with at most $10\eps_2n$ missing cross-edges.
\item If $\bm{x}$ is the unit Perron vector of $H$, 
then $ x_v^2\le (1+O_F(\eps_1))\frac1n$ for every 
$v\in V(H)$. 
\end{enumerate}
\end{claim}

\begin{proof}
Since $W_i=\varnothing$, the definition of $W_i$ gives $d_{V_i}(v)<4\eps_1n$ for every $v\in V_i$.  This proves (i).

Since $S^{(2)}=\varnothing$, every $v\in V_i$ satisfies
$ d_H(v)>\left(1-\frac1r-4\eps_2\right)n$. 
The number of missing cross-neighbors of $v$ is
$
  \sum_{j\ne i}\bigl(|V_j|-d_{V_j}(v)\bigr)
  =n-|V_i|-d_H(v)+d_{V_i}(v)$. 
Using Claim~\ref{clm:balanced-partition}, part (i), and $\eps_1\ll\eps_2$, this is at most
$
  n-\left(\frac1r-3\eps_1\right)n
  -\left(1-\frac1r-4\eps_2\right)n+4\eps_1n
  \le 10\eps_2n$.
This proves (ii).

Recall that $x_{u^\ast}=\max_v \{x_v\}$.  Claim~\ref{clm:perron-almost-constant} gives $ x_u\ge(1-16f^2\eps_1)x_{u^\ast}$  
for every  $u\in \bigcup_{i=1}^r V_i^{(1)}$. 
Also $|S^{(1)}|\le\eps_1n$ by Claim~\ref{clm:S-small}, and $W_i=\varnothing$ implies $V_i^{(1)}=V_i\setminus S^{(1)}$.  Hence
\[
  1=\sum_vx_v^2
  \ge (n-|S^{(1)}|)(1-16f^2\eps_1)^2x_{u^\ast}^2
  \ge (1-O_F(\eps_1))n x_{u^\ast}^2.
\]
Thus, we get $x_{u^\ast}^2\le(1+O_F(\eps_1))/n$, proving (iii).
\end{proof}

The next claim gives an asymptotically sharp local-count lower bound. 
 Let $P=\bigcup_{i=1}^rE(H[V_i])$ and $p=\sum_{i=1}^r e(H[V_i])$ be the set and the number of class-edges of $H$, respectively. 

\begin{claim}[sharp local count]\label{clm:sharp-local-count}
For every class-edge $e\in \bigcup_{i=1}^rE(H[V_i])$, we have 
\begin{equation*} 
  N_F(H,e)\ge \big(1-O_F(\eps_2) \big)c(n,F).
\end{equation*}
\end{claim}

\begin{proof}
Assume $e\in E(H[V_i])$ and write $e=uv$.  Fix once and for all an orientation $(u,v)$ of this host edge.  Let
$  \mathbf n=(n_1,\ldots,n_r)=(|V_1|,\ldots,|V_r|)$. 
Let $K_i(\mathbf n)+e$ be the graph obtained from the complete $r$-partite graph with part sizes $\mathbf n$ by adding the single internal edge $e$ inside the $i$th part, and let $c_i(\mathbf n,F)$ be the number of copies of $F$ in this one-edge graph.  Claim~\ref{clm:balanced-partition} gives 
$ \left|n_j-\frac nr\right|\le 3\eps_1n$ for every $j\in [r]$. 
Therefore, Lemma~\ref{lem:near-balanced-one-edge}, applied with $\xi=3\eps_1$, gives
\begin{equation}\label{eq:one-edge-continuity}
  c_i(\mathbf n,F)
  \ge \big(1-O_F(\eps_1)\big)c(n,F).
\end{equation}

Every copy counted by $c_i(\mathbf n,F)$ uses the edge $e$: after removing $e$, the graph is $r$-partite and therefore contains no copy of $F$.  We now estimate how many of these copies are not present in $H$.  Such a copy must use at least one cross-edge of $K_i(\mathbf n)+e$ that is missing from $H$.

We count bad copies first as labeled edge-preserving injections and then divide by at most the constant $|\Aut(F)|$; this only changes the implicit constant.  By Claim~\ref{clm:clean-case}(ii), each endpoint of $e$ is incident with at most $10\eps_2n$ missing cross-edges.  If a missing required cross-edge is incident with $u$ or $v$, then there are $O_F(\eps_2n)$ choices for that missing host edge, $O_F(1)$ choices for the edge of $F$ mapped to it and for the endpoint assignment, and $O_F(n^{f-3})$ choices for the remaining image vertices.  This gives $O_F(\eps_2n^{f-2})$ bad labeled injections.

It remains to consider a missing required cross-edge whose two endpoints are different from $u$ and $v$.  By Claim~\ref{clm:clean-case}(ii), the total number of missing cross-edges is at most
\[
  \frac12\sum_{w\in V(H)}10\eps_2n=O(\eps_2n^2).
\]
After choosing such a missing host edge and keeping the two fixed endpoints $u,v$ of $e$, there are $O_F(1)$ choices for the corresponding two vertices of $F$ and their assignment to the missing edge, and $O_F(n^{f-4})$ choices for the remaining image vertices.  This gives another $O_F(\eps_2n^{f-2})$ bad labeled injections.  If a copy has several missing required cross-edges, choose one of them by any fixed deterministic rule; the preceding bounds still cover it.  Hence the total number of copies counted by $c_i(\mathbf n,F)$ but not by $N_F(H,e)$ is
$
  O_F(\eps_2n^{f-2})$. 
Using $\eps_1\ll\eps_2$ and $c(n,F)=\alpha_Fn^{f-2}+O_F(n^{f-3})$, 
we see that \eqref{eq:one-edge-continuity} yields
$N_F(H,e)
  \ge \big(1-O_F(\eps_2) \big)c(n,F)$, 
as required.
\end{proof}

\subsection{Conversion of the spectral gap into class-edges}

The following is the main first-order refinement.  
It uses the almost-uniform Perron vector in the clean case and keeps the negative term in the expansion of $\tau_r(h-p)$.

\begin{claim}[sharp conversion from spectral gap to class-edges]\label{clm:sharp-class-edges}
We have 
\begin{equation}\label{eq:sharp-p}
  p\ge (1-O_F(\theta)-o(1))Cn.
\end{equation}
\end{claim}

\begin{proof}
Write $H^*=H-P$ for the subgraph of $H$ obtained by removing all $p$ class-edges.  Since $H^*$ is $r$-partite with $h-p$ edges, Theorem~\ref{thm:nikiforov} gives $\lambda(H^*)\le\tau_r(h-p)$.  Let $\bm{x}$ be the unit Perron vector of $H$.  Splitting $A(H)=A(H^*)+A(P)$, using $\bm{x}^{T}A(H^*)\bm{x}\le\lambda(H^*)$ and the bound $x_ux_v\le(1+O_F(\eps_1))/n$ from Claim~\ref{clm:clean-case}(iii), 
we get 
\[
  \lambda(H)=\bm{x}^{T}A(H)\bm{x}
  =\bm{x}^{T}A(H^*)\bm{x}+2\!\sum_{uv\in P}x_ux_v
  \le\tau_r(h-p)+2(1+O_F(\eps_1))\tfrac pn.
\]
Together with \eqref{eq:surviving-gap}, this yields
\begin{equation}\label{eq:gap-to-p-start}
  (1-\theta)C\le\tau_r(h-p)-\tau_r(h)+2(1+O_F(\eps_1))\tfrac pn.
\end{equation}
If $p\ge Cn$, then \eqref{eq:sharp-p} is immediate, so assume $p<Cn=O(n)$.  Then $p/h=o(1)$, since $h=\Theta_r(n^2)$ by Claim~\ref{clm:n-size}, which also gives $h=(1-\frac{1}{r}) \frac{n^2}{2} (1+O_F(\eps_1))$ and  $\tau_r(h)=(1-\tfrac1r)n(1+O_F(\eps_1))$.  Note that $\tau_r(h)^2-\tau_r(h-p)^2=2(1-\tfrac1r)p$, while $\tau_r(h-p)=\tau_r(h)\sqrt{1-p/h}=\tau_r(h)(1+o(1))$. Hence
\[
  \tau_r(h)-\tau_r(h-p)
  =\frac{2(1-\tfrac1r)p}{\tau_r(h)+\tau_r(h-p)}
  =\big(1+O_F(\eps_1)+o(1)\big)\tfrac pn.
\]
Substituting into \eqref{eq:gap-to-p-start} gives $(1-\theta)C\le(1+O_F(\eps_1)+o(1))\tfrac pn$. 
Since $\eps_1\ll\theta$, we have 
\[
  p\ge\frac{1-\theta}{1+O_F(\eps_1)+o(1)}\,Cn\ge(1-O_F(\theta)-o(1))Cn,
\]
which proves \eqref{eq:sharp-p}.
\end{proof}

Finally, we count the copies of $F$.  
The families of copies counted by $N_F(H,e)$, for distinct class-edges $e\in P$, are disjoint, since each counted embedding has a unique within-class image-edge, namely $e$. Thus, we obtain $ N_F (G)\ge N_F (H)\ge\sum_{e\in P}N_F(H,e)$.  By Claims~\ref{clm:sharp-local-count} and~\ref{clm:sharp-class-edges}, together with Mubayi's estimate $c(n,F)=\alpha_Fn^{f-2}(1+o(1))$ and $\eps_2\ll\theta$, we have 
\[
   N_F (G)\ge(1-O_F(\eps_2)-o(1))\,p\,c(n,F)
  \ge(1-O_F(\theta)-o(1))\,\alpha_F\,Cn^{f-1}.
\]
Since $n^{f-1}\ge(1-O_F(\eps_1+\delta))\big(\tfrac{2mr}{r-1}\big)^{(f-1)/2}$ by Claim~\ref{clm:n-size} and $h>(1-\delta)m$, 
we get 
\[
   N_F (G)\ge(1-O_F(\theta)-o(1))\,\alpha_F\Big(\tfrac{2r}{r-1}\Big)^{(f-1)/2}C\,m^{(f-1)/2}.
\]
The hierarchy makes the fixed $O_F(\theta)$ error less than $\eta/2$, and then large $m$ makes the $o(1)$ error less than $\eta/2$; this contradicts \eqref{eq:sharp-counter} and proves the lower bound in Theorem~\ref{thm:linear-small-gap}.

\subsection{Sharpness of the exact coefficient}\label{sec:sharpness}

We prove the tightness of  the bound of Theorem~\ref{thm:linear-small-gap}, showing that the coefficient $\Ksharp:=\alpha_F (\tfrac{2r}{r-1})^{(f-1)/2}$ cannot be improved in the linear small-gap limit. 
Fix $\eta>0$ and a small $\xi>0$, and let $C>0$ be small with $t:=(1+\xi)C$.  Take $n=ra$, start from the balanced complete $r$-partite graph with parts $V_1,\ldots,V_r$ of size $a$, and add a matching of size $q:=\lfloor tn\rfloor$ inside $V_1$; this is possible once $tr<1/2$, since then $q\le a/2=|V_1|/2$.  Write $Y_{n,r,q}$ for the resulting graph and $m=e(Y_{n,r,q})$.

\begin{claim}[spectral gap via adding a matching]\label{clm:matching-gap}
As $n\to\infty$ with $t$ fixed and small, we have $\lambda(Y_{n,r,q})-\tau_r(e(Y_{n,r,q}))=t+o(1)$; in particular $\lambda(Y_{n,r,q})\ge\tau_r(m)+C$ for all large $n$.
\end{claim}

\begin{proof}
Partition the vertices into the $2q$ matched vertices of $V_1$, the $a-2q$ unmatched vertices of $V_1$, and $V_2\cup\cdots\cup V_r$.  This partition is equitable, with quotient matrix
\[
  Q=
  \begin{pmatrix}
    1 & 0 & (r-1)a \\
    0 & 0 & (r-1)a \\
    2q & a-2q & (r-2)a
  \end{pmatrix},
\]
and $Y_{n,r,q}$ is connected, so $\lambda(Y_{n,r,q})$ equals the largest eigenvalue of $Q$.  Write $\lambda(Y_{n,r,q})=(r-1)a+\varepsilon$ for some real $\varepsilon$ determined later. 
Since $Y_{n,r,q}$ contains $T_{n,r}$ as a subgraph and 
 the maximum degree of $Y_{n,r,q}$ is $(r-1)a+1$, we have $0\le\varepsilon\le1$.  Expanding the characteristic polynomial,
\[
  0=\det\!\big(Q-((r-1)a+\varepsilon)I\big)
  =2a(r-1)q-a^2r(r-1)\varepsilon+a\big(r\varepsilon-(2r-1)\varepsilon^2\big)+\varepsilon^2-\varepsilon^3.
\]
With $q=\lfloor tn\rfloor=tra+O(1)$ and $0\le\varepsilon\le1$, this reads $0=a^2r(r-1)(2t-\varepsilon)+O_r(a)$, so $\varepsilon=2t+O_r(a^{-1})=2t+o(1)$.  On the other hand $e(Y_{n,r,q})=e(T_{n,r})+q=\tfrac{r-1}{2r}n^2+tn+O(1)$, whence
\[
  \tau_r(e(Y_{n,r,q}))
  =(1-\tfrac1r)n\sqrt{1+\tfrac{2t}{(1-1/r)n}+O(n^{-2})}
  =(1-\tfrac1r)n+t+o(1)=(r-1)a+t+o(1).
\]
Subtracting gives 
$$\lambda(Y_{n,r,q})-\tau_r(e(Y_{n,r,q}))=t+o(1).$$  
Since $t=(1+\xi)C>C$, 
 the final assertion $\lambda (Y_{n,r,q}) \ge \tau_r(m) +C$ follows.
\end{proof}

\begin{claim} \label{clm:matching-count}
As $n\to\infty$ with $t$ fixed and small,  we have 
$$ N_F (Y_{n,r,q})=(1+o(1))\,q\,c(n,F)=(1+o(1))\alpha_F t n^{f-1}.$$
\end{claim}

\begin{proof}
Let $\mathcal M$ be the set of $q$ matching edges, and call $e\in\mathcal M$ \emph{active} for an embedding $\phi\colon F\hookrightarrow Y_{n,r,q}$ if $e$ is the image of an edge of $F$.  As $T_{n,r}$ is $F$-free, every embedding has at least one active edge.

If an embedding has exactly one active edge $e\in\mathcal M$, then all other edges of $F$ map to cross-edges of $T_{n,r}$, so the embedding already lies in $T_{n,r}+e$ (extra matching edges spanned by the image are irrelevant, as copies are non-induced); hence there are at most $q\,c(n,F)$ such copies.  For embeddings with at least two active edges, pick two of them in $O(q^2)$ ways; since $\mathcal M$ is a matching these have four distinct endpoints, and after fixing the corresponding two edges of $F$ ($O_F(1)$ choices) the remaining vertices admit $O_F(n^{f-4})$ choices, giving $O_F(q^2n^{f-4})=O_F(t^2n^{f-2})=o(tn^{f-1})$ copies.  With $q=tn+O(1)$ and $c(n,F)=\alpha_Fn^{f-2}+O_F(n^{f-3})$, this gives the upper bound 
$$ N_F (Y_{n,r,q})\le q\,c(n,F)+O_F(q^2n^{f-4})=(1+o(1))\alpha_F t n^{f-1}.$$

For the lower bound, every copy of $F$ in $T_{n,r}+e$ has $e$ as its unique active edge; viewed inside $Y_{n,r,q}$ the edge $e$ stays active and no other matching edge becomes active, since the edge images are fixed by the embedding.  Thus the copies from distinct $e\in\mathcal M$ are distinct and all survive in $Y_{n,r,q}$, so $ N_F (Y_{n,r,q})\ge q\,c(n,F)$.  Combining the two bounds proves the claim.
\end{proof}

Since $m\!=\!\tfrac{r-1}{2r}n^2+O_C(n)$, 
we get $n^{f-1}\!=\! (1\!+\! o(1)) (\tfrac{2mr}{r-1} )^{(f-1)/2}$. 
 Claim~\ref{clm:matching-count} gives 
\[
   N_F (Y_{n,r,q})
  \le(1+o(1))\alpha_F t\Big(\tfrac{2r}{r-1}\Big)^{(f-1)/2}m^{(f-1)/2}
  =(1+o(1))(1+\xi)\Ksharp\,C m^{(f-1)/2}.
\]
Given $\eta>0$, choose $\xi$ with $(1+\xi)(1+\eta/3)\le1+\eta$, then $C_1$ small enough that, for all $0<C\le C_1$, the value $t=(1+\xi)C$ lies in the small range required by the claims and satisfies $tr<1/2$, and finally $n$ large enough that the factor $1+o(1)$ above is at most $1+\eta/3$.  This gives $ N_F (Y_{n,r,q})\le(1+\eta)\Ksharp\,C m^{(f-1)/2}$, while Claim~\ref{clm:matching-gap} gives $\lambda(Y_{n,r,q})\ge\tau_r(m)+C$.  This proves the sharpness in Theorem~\ref{thm:linear-small-gap}, and combined with the lower bound it also proves the limit identity~\eqref{eq:exact-limit}.

\begin{remark}
\label{rem:fixed-gap-caveat}
The coefficient of Theorem~\ref{thm:linear-small-gap} is a linear term as $C\to 0^+$.
 For a fixed positive gap, the sharp coefficient need not be linear in $C$.  
 Let $T_{n,r,q}$ be the graph obtained by adding a star with $q:=tn$ edges inside one part. 
 A similar calculation gives $\lambda (T_{n,r,q}) -\tau_r(m)=t+\tfrac{r}{r-1}t^2+o(1)$. Solving $\tfrac{r}{r-1}t^2 +t =C$ for $t$, the internal-edge density required to produce a fixed gap $C$ satisfies  
\[
  t=t_r(C)=\tfrac{r-1}{2r}\Big(\sqrt{1+\tfrac{4rC}{r-1}}-1\Big)
  =C-\tfrac{r}{r-1}C^2+O_r(C^3),
\]
which is nonlinear in $C$ and strictly less than the value $C$. 
Consequently, $N_F(T_{n,r,q}) = (1+o(1))tn\cdot c(n,F)$, which is nonlinear in $C$. Thus, determining the  coefficient for every fixed $C>0$ amounts to a finer nonlinear extremal problem for the added internal graph.
\end{remark}

\section{Concluding remarks}

\label{sec:concluding} 

We summarize three distinct lines of supersaturation results. In the classical line, one counts copies of $F$ once the size exceeds the
Tur\'an number, from Erd\H{o}s for triangles to Mubayi~\cite{Mub2010} for all color-critical graphs. In
the vertex-spectral line, one counts copies once $\lambda(G)$ exceeds
the threshold $(1- \frac{1}{r})n$; here Bollob\'as and Nikiforov~\cite{BN2007jctb}
counted cliques, and Fang, Li, Lin and Ma~\cite{FLLM2025} proved a vertex-spectral version of Mubayi's theorem. 
The edge-spectral line has  been developed by 
Fang, Lin and Zhai~\cite{FLZ} in the threshold case $\lambda^2(G) > (1- \frac{1}{r})2m$. 
The present paper settles the small-gap regime of the edge-spectral line under an additive gap $\lambda^2(G) > (1- \frac{1}{r})2m +q$ for every $0< q\le \delta_F \sqrt{m}$. This solves a conjecture of Fang, Lin and Zhai, and  gives an edge-spectral version of Mubayi's theorem. 

\smallskip 
We close with some related problems and possible directions for future work. 

\paragraph{The pointwise constant for a fixed gap.}
For a fixed real number $C>0$, we define 
\[
  g_F(C):=\liminf_{m\to\infty}\ \frac{1}{m^{(f-1)/2}}
  \min\bigl\{\,N_F(G): e(G)=m,\ \lambda(G)\ge \sqrt{(1\!-\!1/r)2m}+C\,\bigr\}.
\]
Theorem~\ref{thm:linear-small-gap} says that $g_F(C)=\kappa_F\,C\,(1+o_C(1))$ as
$C\to 0^+$, so the matching-added construction $Y_{n,r,q}$ is optimal to first order; see Section \ref{sec:sharpness}. 
However, for a fixed large $C>0$,
 adding a matching is no longer optimal: by Remark~\ref{rem:fixed-gap-caveat}, the construction $T_{n,r,q}$, 
adding a star with $t n$ edges inside one part,  lifts the spectral radius by
$t+\frac{r}{r-1}t^2+o(1)$, so the same gap $C$ is reached with
$t=t_r(C)=C-\frac{r}{r-1}C^2+O_r(C^3)<C$ internal edges, and hence with strictly
fewer copies. Thus, we have 
$ (1-\eta(C))\,\kappa_F\,C \le g_F(C) \le\kappa_F\,t_r(C) < \kappa_F\,C$ for $0<C\le C_0$, 
and the exact value is open.

\begin{problem}\label{prob:pointwise}
Determine $g_F(C)$ for every real number $C>0$. Equivalently, find the graph that minimizes the
number of copies of $F$ among $m$-edge graphs with $\lambda(G)\ge \sqrt{(1\!-\!1/r)2m}+C$.
\end{problem}

We expect the extremal graph to be a Tur\'an graph $T_{n,r}$ together with
an optimal ``internal'' graph $D$ added inside one part. Since each internal edge
creates about $c(n,F)$ copies, while the spectral lift per edge grows as $D$ becomes
more concentrated, Problem~\ref{prob:pointwise} reduces to a clean extremal question:
for a target lift $C$, which $D$ uses the fewest edges (and the fewest extra copies
coming from two internal edges)? It is natural to ask whether the optimal $D$ is a
clique. We note the following contrast: for $r=2$ and $F=K_3$, Chen, Li and Tang~\cite{CLT2026} showed that the analogous count is \emph{exactly} linear in the
gap, with split graphs extremal. The nonlinear correction above is therefore a new
feature of the case $r\ge 3$. It would be interesting to decide for which $F$ the
function $g_F$ is linear.

\paragraph{Further directions.} 
 When $\chi(F)=3$, the extremal graphs for
the edge-spectral Tur\'an problem are no longer balanced complete bipartite graphs but
are often split graphs, as shown in~\cite{ZLS2021,LZS2024, LZZ2024, LLLY2026}.
The edge-spectral supersaturations for triangles and books were studied in~\cite{NZ2021, ZLL2026, ZZ2026, CLT2026}, but the general case is open: 
establish a sharp edge-spectral supersaturation theorem for color-critical graphs $F$
with $\chi(F)=3$, in particular for the odd cycles $C_{2k+1}$ with $k\ge 2$. Here, the
candidate extremal graphs are split graphs rather than Tur\'an graphs, so a different
stability analysis is needed.

Four further directions seem worthwhile. First, beyond color-critical graphs the minimum
supersaturation configurations are more involved already in the combinatorial setting
(Ma and Yuan~\cite{MY2025}); finding the edge-spectral count for a general graph $F$
just above its spectral threshold is open. Second, for some bipartite $F$, the threshold is
$\lambda(G)\approx\sqrt m + O(1)$, and sharp counts are tied to spectral Sidorenko-type
inequalities \cite{LLLZ2026}; 
only the threshold case above the split graphs is known, and the general case with an additive spectral gap remains unknown. Third, one may replace the adjacency spectral radius by the
$p$-spectral radius \cite{KN2014,KNY2015} or 
 the signless Laplacian spectral radius \cite{ZLL2026LAA, ZLF2026} and ask for the
corresponding sharp supersaturation. 
Finally, removing the $o(1)$ error and obtaining an \emph{exact} count, in the spirit
of the work of Lov\'asz--Simonovits and Liu--Pikhurko--Staden, would already be of interest for the clique $K_{r+1}$,
where the constant $\alpha_F=(1/r)^{r-1}$ is explicit.

\end{document}